\documentclass[12pt]{amsart}
\usepackage{amssymb, latexsym}

\theoremstyle{plain}
\newtheorem{theorem}{Theorem}[section]
\newtheorem{corollary}[theorem]{Corollary}
\newtheorem{proposition}[theorem]{Proposition}
\newtheorem{lemma}[theorem]{Lemma}

\theoremstyle{definition}
\newtheorem{definition}[theorem]{Definition}
\newtheorem{remark}[theorem]{Remark}
\newtheorem{example}[theorem]{Example}

\numberwithin{equation}{section}

\usepackage{dsfont}
\usepackage{braket}
\usepackage{color}
\usepackage{tikz}
\usepackage{tkz-euclide}
\usepackage{tikz-3dplot}
\usepackage{graphicx}
\usepackage{hyperref}
\usepackage[color,matrix,arrow]{xy}
\input xy
\xyoption{all}

\begin{document}

\title[Convex polygon coordinates]{The inverse problem of convex polygon coordinates}

\author[Romanowska]{A.B. Romanowska}
\author[Smith]{J.D.H. Smith}
\author[Zamojska-Dzienio]{A. Zamojska-Dzienio}
\address{(A.R., A.Z.) Faculty of Mathematics and Information Science\\
Warsaw University of Technology\\
00-662 Warsaw, Poland}

\address{(A.R., J.S.) Department of Mathematics\\
Iowa State University\\
Ames, Iowa, 50011, USA}

\email{(A.R.) anna.romanowska@pw.edu.pl\phantom{,}}
\email{(J.S.) jdhsmith@iastate.edu\phantom{,}}
\email{(A.Z.) anna.zamojska@pw.edu.pl}

\urladdr{(J.S.) 
\protect
{
\href{https://jdhsmith.math.iastate.edu/}
{https://jdhsmith.math.iastate.edu/}
}
}

\begin{abstract}
Each convex combination of extreme points of a compact convex set represents a certain point of the convex set. Barycentric coordinates provide solutions to the inverse problem of expressing an element of a compact convex set as a convex combination of a finite number of extreme points of the set. Various approaches to this problem have arisen, in various contexts. The most general solution, namely the Gibbs coordinates based on entropy maximization, actually work in the broader setting of barycentric algebras, which constitute semilattice-ordered systems of convex sets. These coordinates involve exponential functions. For convex polytopes, Wachspress coordinates offer solutions which only involve rational functions. The current paper is primarily focused on convex polygons in the plane. After summarizing the Gibbs and Wachspress coordinates, we identify where they agree, and provide comparisons between them when they do not. With an example, we also show how Gibbs coordinates of a polygon with rational vertices may be construed as algebraic functions.
\end{abstract}

\thanks{The research of the third author was partially conducted while a Visiting Scholar at Iowa State University. It was supported under a Fulbright Senior Award granted by the Polish-U.S. Fulbright Commission, and additionally by the Kosciuszko Foundation---The American Centre of Polish Culture.}

\maketitle

\tableofcontents

\section{Introduction}

\subsection{Motivation}

Each point of a simplex has a unique expression as a convex combination of vertices of the simplex (\S\ref{SS:volmtric}). More generally, each point $x$ in the convex hull of a set $V$ of points may be expressed as a convex combination of certain points of $V$.\footnote{We regard this statement as tautological, by defining the convex hull of $V$ as the set of convex combinations of elements of $V$.} As elements of the closed unit interval $I$, the coefficients or \emph{weights} in these convex combinations constitute \emph{barycentric coordinates} of the point $x$ with respect to $V$. In the general situation, a single point $x$ may be recovered from various distinct expressions as a convex combination of points of $V$. For example, the barycenter of a square may be expressed as the midpoint of either of the two diagonals. This paper is concerned with the problem of determining a particular set of barycentric coordinates for each point $x$ of the convex hull of a set $V$ of points, primarily in the case where $V$ is the set of vertices or extreme points of a convex polygon.

Barycentric coordinates have been studied widely within the geometric literature, typically in response to the demands of numerical analysis and computer graphics. The current paper brings an algebraic perspective to the problem, based on \emph{barycentric algebras} which provide intrinsic descriptions of convex sets, semilattices, and more general semilattice-ordered systems of convex sets, independently of any ambient affine or vector space. The relevance of this algebraic perspective may be illustrated with a quotation from a paper published in 2007:
\begin{quote}
``Despite much work in the discrete 2D case, no explicit formulation of
barycentric coordinates for convex polytopes valid in arbitrary dimension is currently
available''
\end{quote}
\cite[\S1.4]{WarSchHirDes}. In fact, the Gibbs coordinates discussed below had already been introduced for any finitely-generated barycentric algebra five years earlier \cite[Ch.~IX]{Modes}, as part of a \emph{hierarchical statistical mechanics}  designed to handle complex systems functioning on multiple levels.

\subsection{Plan of the paper}

Section~\ref{S:BarycAlg} provides a review of those aspects of barycentric algebras which are pertinent to the coordinate systems discussed here. Barycentric algebras are presented as sets equipped with a collection of binary operations, indexed by the open real unit interval $I^\circ$. The operations are required to satisfy \emph{idempotence} \eqref{E:idemptnc}, \emph{skew commutativity} \eqref{E:skewcomm}, and \emph{skew associativity} \eqref{E:skewasoc} axioms. Section~\ref{S:BarycAlg} concludes with a brief dictionary that reconciles the differing terminologies for common concepts that have arisen in different areas. As an example, simplices are finitely generated free barycentric algebras (\S\ref{SSS:fralasmp}).

The Gibbs coordinates are presented in Section~\ref{S:GibbsBAl}. When restricted to convex polytopes, they are the probability weights appearing in the \emph{Gibbs} or \emph{Boltzmann distributions}, also described by the softmax function. The randomness of a finite probability distribution $(p_1,\dots,p_n)$ is measured by its \emph{entropy} $-\sum_{i=1}^np_i\log p_i$ (where $0\log0$ is interpreted as $0$). Amongst all expressions of a point $x$ as a convex combination with weights $(p_1,\dots,p_n)$ of points from a set $V$ of cardinality $n$, the Gibbs coordinates are the unique maximizers of the entropy (compare \cite{HormSuku}). In general, the Gibbs coordinates involve exponential functions. Nevertheless, when the points of the set $V$ are mutually related by rational convex combinations, the technique we introduce in \S\ref{E:equaGibs} may be used to handle Gibbs coordinates with algebraic functions, avoiding the transcendental exponentials.

The \emph{Wachspress coordinates} surveyed in Section~\ref{S:Wchspres} form another well-established system of barycentric coordinates \cite{Wachspress1,Wachspress2}. Since Wachspress coordinates only involve rational functions, they are amenable to techniques of algebraic geometry and commutative algebra \cite{IrvSch,KohRan}, although we avoid any discussion of that aspect in this paper. While the development of Wachspress coordinates has progressed through a series of stages recorded in the literature (such as \cite{FloaterWMVC, FloaterGBCA, KosBarW, KosBar, Skala}, \cite{Warren} -- \cite{WarSchHirDes}), our survey enjoys the benefit of hindsight. In particular, given the appearance of the boundary curvature in the formula \eqref{E:CntWxpWt} for Wachspress weights on strictly convex sets \cite{WarSchHirDes}, we emphasize the interpretation \cite[Fig.~1]{WarSchHirDes} of a determinant in Wachspress weights on a polygon as a discrete analogue of boundary curvature at a vertex of the polygon (\S\ref{SSS:WaWeSuFu}). Our casting of Wachspress coordinates as a bulk-boundary correspondence may be new (\S\ref{SSS:BulkBdry}).

Section~\ref{S:WchsGibs} presents a comparison of Gibbs and Wachspress coordinates. Theorem~\ref{T:WaGiSmSi} shows that they agree on \emph{semisimplices} --- polytopes which decompose as direct products of simplices when construed as barycentric algebras.
On the other hand, Example~\ref{X:WxNeqGbs} exhibits a quadrilateral where the Gibbs and Wachspress coordinates may differ. For a given $n$-gon $\Pi$ in the plane, the (\emph{G$-$W}) \emph{discrepancy field} is introduced in Definition~\ref{D:GWdscrpFld} as an $(n-1)$-dimensional vector field over $\Pi$ whose $i$-th coordinate, for $1\le i<n$, measures the difference between the $i$-th Gibbs coordinate and the $i$-th Wachspress coordinate. Proposition~\ref{P:GWdscrpFld} shows that the discrepancy vectors lie in a vector space of dimension $n-3$. Thus for quadrilaterals, the comparison between the Gibbs and Wachspress coordinates is adequately captured by the norm of the discrepancy vector, displayed as a contour plot in Figure~\ref{F:QrCrtGbs} for the quadrilateral of Example~\ref{X:WxNeqGbs}. The Gibbs and Wachspress coordinates agree on the boundary of the quadrilateral. They also agree on the \emph{equator}, an algebraic curve connecting two opposite vertices, whose equation is presented in Theorem~\ref{T:EqtrEqun}.

Section~\ref {S:ConcFuWk} offers some concluding remarks and pointers towards future research directions initiated by the approach of this paper.

\subsection{Notational conventions}

The essence of algebra lies in the repeated concatenation of successive functions, where the output of one function is fed as an input to the next. In this context, Euler's analytical notation as in the expression $\sin x$, where the argument is read after the function, clashes with the natural direction of reading from left to right.\footnote{Even within texts written in Semitic languages, mathematics is read in this direction.} Thus, while paying due respect to many traditional analytical notations, such as $H(\alpha)$ for the entropy \eqref{E:entalpha} and $Z(\beta)$ for the partition function \eqref{E:PartyFun}, our default option for function notation is \emph{algebraic} or \emph{diagrammatic}, with functions following their arguments (e.g., $n!$ for the factorial function), or placed as a suffix (e.g., $x^2$ for the squaring function). A fortunate consequence of this convention in linear algebra is that vectors (to be transformed by multiplication with a matrix following) appear as rows or tuples, and do not require transposition to fit on a line.

The expression $x\ll y$ is used to denote that a positive real number $x$ is not only less than a real number $y$, but also very small by comparison. Logarithms that are presented without an explicit base are natural. For other notational conventions and definitions that are not otherwise specified explicitly (e.g., for group actions), readers are invited to consult \cite{PMA}.

\section{Barycentric algebras}\label{S:BarycAlg}

\subsection{Operations on real numbers and binary digits}\label{SS:operreal}

The algebra of real numbers is traditionally performed in terms of field computations like the addition $p+q$ and division $q/p$ of real numbers $p$ and $q$, with $p\ne 0$ for the division. The barycentric algebra discussed in this paper requires \emph{operations} which are defined for all arguments, specializing to more familiar Boolean operations on the subset $\set{0,1}$ of the reals. In fact, this specialization will also work in any field. In particular, it works if the set $\set{0,1}$ of binary digits is interpreted as the two-element (Galois) field $\mathsf{GF}(2)$ or field of integers modulo $2$.

For a real number $p$, define the \emph{complementation}
$
p'=1-p
$
specializing to the Boolean $\neg p$ or $\textsc{not}\ p$ on the set $\set{0,1}$ of binary digits. Note that the complementation is \emph{involutive}:
$
p''=p \, .
$
For real numbers $p$ and $q$, define the \emph{product}
\begin{equation}\label{E:prodand}
p\cdot q=pq
\end{equation}
specializing to the Boolean $\wedge$ or $\textsc{and}$ on $\{0,1\}$. Define the \emph{dual product}
\begin{equation}\label{E:duprodor}
p\circ q = p+q-pq
\end{equation}
specializing to the Boolean $\vee$ or (non-exclusive) $\textsc{or}$ on $\set{0,1}$. Note that the dual product may be defined in terms of the product and complementation using \emph{De Morgan's law}
$
p\circ q=(p'q')'
$
or $(p\circ q)'=p'q'$. Define the \emph{implication}\footnote{In order to view a valid fraction as an implication, the lexicographically ordered mnemonic ``denominator implies numerator'' may be helpful.}
\begin{equation}\label{E:implicat}
p\to q=\mathbf{if}\ (p=0)\ \mathbf{then}\ 1\ \mathbf{else}\ \frac qp
\end{equation}
specializing to the Boolean implication $p\to q=(\neg p)\vee q$ on $\{0,1\}$. Note that, unlike the division $q/p$, the implication (\ref{E:implicat}) is always defined in any field.

\subsection{Binary operations}\label{SS:binopers}

When $x$ and $y$ are elements of a real vector space, and $p$ is a real number, it is convenient to define
\begin{equation}\label{E:punderln}
xy\,\underline{p}=x(1-p)+yp=xp'+yp \, ,
\end{equation}
so that $\underline{p}$ is understood as a binary operation combining the arguments $x$ and $y$. Schematically, the binary operation may be understood as a circuit element or ``black box'' combining the inputs $x$ and $y$ to produce the output $xy\,\underline{p}$. For a second real number $q$ and vector $z$, one may concatenate circuit elements to yield
\begin{equation}\label{E:compoper}
xy\,\underline{p}\, z\,\underline{q}
=(xp'+yp)z\underline{q}
=xp'q'+ypq'+zq \, .
\end{equation}
Alternatively one may concatenate the circuit elements to yield
\begin{equation}\label{E:opercomp}
x\,yz\,\underline{p}\,\,\underline{q}
=x(yp'+zp)\underline{q}
=xq'+yp'q+zpq \, .
\end{equation}
Note that the parsing of the left hand sides
of (\ref{E:compoper}) and (\ref{E:opercomp})
is unique, even without the insertion of any brackets. This is one of the many advantages of the algebraic notation (\ref{E:punderln}) for binary operations.
\begin{itemize}
\item
If $p$ is any real number, then the operation (\ref{E:punderln}) makes sense when the inputs $x$ and $y$ lie in a real affine space $A$, for example a real vector space where the special role of the origin $0$ is surpressed.
\item
If $p$ is an element of the closed real unit interval
\begin{equation}\label{E:closedIn}
I=[0,1]=\{p\mid 0\le p\le 1\} \, ,
\end{equation}
then the operation (\ref{E:punderln}) makes sense when the inputs $x$ and $y$ lie in some convex set $C$, for example some interval on the real line.
\item
If $p$ is a binary digit $0$ or $1$, the operation $xy\,\underline{p}$ makes sense as
$$
xy\,\underline{p}=
\mathbf{if}\ (p=0)\ \mathbf{then}\ x\ \mathbf{else}\ y
$$
when the inputs $x$ and $y$ are elements of some arbitrary set $S$.
\end{itemize}
Given an arbitrary set $S$, consider the convex set $SB$\footnote{This notation is taken from \eqref{E:FBAlgFct} below.} of all finite probability distributions on $S$, identifying each element $x$ of $S$ with the distribution putting weight $1$ on $x$. For elements $x$ and $y$ of $S$, and $p$ in $I$, the operation (\ref{E:punderln}) produces the distribution selecting $y$ with probability $p$ and $x$ with probability $p'$.

\subsection{Barycentric algebras}\label{SS:barycent}

\subsubsection{The basic definition}

\begin{definition}\cite{RS85,RS90,Modes}
A \emph{barycentric algebra} $A$ or $(A,I^\circ)$ is defined as a set $A$ that is equipped with binary \emph{operations}
\begin{equation}\label{E:OpOfBaAl}
\underline{p}\colon A\times A\to A;(x,y)\mapsto xy\,\underline{p}
\end{equation}
for each element (or \emph{operator}) $p$ of the open real unit interval
\begin{equation}\label{E:openItvl}
I^\circ=]0,1[=\{p\mid 0<p<1\} \, ,
\end{equation}
the interior of the closed interval \eqref{E:closedIn}. The operations \eqref{E:OpOfBaAl} are required to satisfy the properties of \emph{idempotence}
\begin{equation}\label{E:idemptnc}
xx\,\underline{p}=x
\end{equation}
for $x$ in $A$, \emph{skew-commutativity}
\begin{equation}\label{E:skewcomm}
xy\,\underline{p}=yx\,\underline{p'}
\end{equation}
for $x$, $y$ in $A$, and \emph{skew-associativity}
\begin{equation}\label{E:skewasoc}
xy\,\underline{p}\,z\,\underline{q}
=x\,yz\,\big(\,\underline{p\circ q\to q}\,\big)\,\underline{p\circ q}
\end{equation}
for $x$, $y$, $z$ in $A$.
\end{definition}

\subsubsection{Basic examples of barycentric algebras}

\begin{example}\label{X:CoAfBAlg}
When the barycentric algebra operations are interpreted by \eqref{E:punderln}, each convex set (and in particular, each affine space) forms a barycentric algebra.
\end{example}

\begin{example}\label{X:SemiLatt}
A \emph{semilattice} is a semigroup $(S,\cdot)$ which is commutative and idempotent. When each barycentric operation is interpreted as $xy\,\underline p=x\cdot y$ on a semilattice $S$, the semilattice becomes a barycentric algebra. In this case, skew-commutativity and skew-associativity reduce respectively to genuine commutativity and associativity.
\end{example}

\begin{example}\label{X:ExtReals}
Consider the set $\mathbb R^\infty=\mathbb R\cup\set\infty$ of real numbers extended by $\infty$. Take $\mathbb R$ with the barycentric algebra structure defined in Example~\ref{X:CoAfBAlg}. Then, define $x\infty\underline p=\infty$ for all $x\in\mathbb R^\infty$ and $p\in I^\circ$. The set $\mathbb R^\infty$ of extended reals becomes a barycentric algebra.
\end{example}

\subsubsection{Which barycentric algebras are convex sets?}

In Example~\ref{X:CoAfBAlg}, it was noted that convex sets form barycentric algebras. A barycentric algebra $(A,I^\circ)$ forms a convex set iff the \emph{cancellativity} property
$$
\forall\ p\in I^\circ\,,\
\forall\ x,y,z\in A\,,\
xy\,\underline p=xz\,\underline p
\
\Rightarrow
\
y=z
$$
holds \cite{N70}, \cite[Th.~269]{RS85}, \cite[Th.~5.8.6]{Modes}.

\begin{example}
The set $\set{0,1}$ of binary digits in $\mathbb R$ forms a semilattice, interpreted as a barycentric algebra by Example~\ref{X:SemiLatt}. Here, cancellativity fails, since $0\cdot 0=0=0\cdot 1$.
\end{example}

\subsubsection{The entropic property}

\begin{lemma}\label{L:entropic}
A barycentric algebra $(A,I^\circ)$ satisfies the \emph{entropic} property
\begin{equation}\label{E:entropic}
(st\,\underline p)(uv\,\underline p)\underline q
=
(su\,\underline q)(tv\,\underline q)\underline p
\end{equation}
for $s,t,u,v\in A$ and $p,q\in I^\circ$.
\end{lemma}

\begin{proof}
If the proof were for a semilattice $(S,\cdot)$, as in Example~\ref{X:SemiLatt}, it would proceed as follows
\begin{align*}
(s\cdot t)\cdot(u\cdot v)
&
\overset{\mathrm A}=
s\cdot\big[t\cdot(u\cdot v)\big]
\overset{\mathrm A}=
s\cdot\big[(t\cdot u)\cdot v\big]
\\
&
\overset{\mathrm C}=
s\cdot\big[(u\cdot t)\cdot v\big]
\overset{\mathrm A}=
s\cdot\big[u\cdot(t\cdot v)\big]
\overset{\mathrm A}=
(s\cdot u)\cdot(t\cdot v)
\end{align*}
using associativity (A) and commutativity (C) at the indicated steps. The proof for general barycentric algebras proceeds similarly, instead using skew-associativity and skew-commutativity at the corresponding steps.
\end{proof}

\subsection{Subalgebras and homomorphisms}

\subsubsection{Subalgebras, walls, and sinks}

\begin{definition}\label{D:SbWaSink}
Suppose that $(A,I^\circ)$ is a barycentric algebra, and $B$ is a subset of $A$.
\begin{enumerate}
\item[$(\mathrm a)$]
If
$$
\forall\ p\in I^\circ\,,\
x\in B
\mbox{ and }
y\in B
\
\Rightarrow
\
xy\,\underline p\in B\,,
$$
then the subset $B$ is a \emph{subalgebra} of $(A,I^\circ)$.
\item[$(\mathrm b)$]
If
$$
\forall\ p\in I^\circ\,,\
x\in B
\mbox{ and }
y\in B
\
\Leftrightarrow
\
xy\,\underline p\in B\,,
$$
then the subset $B$ is a \emph{wall} of $(A,I^\circ)$.
\item[$(\mathrm c)$]
If
$$
\forall\ p\in I^\circ\,,\
x\in B
\mbox{ or }
y\in B
\
\Rightarrow
\
xy\,\underline p\in B\,,
$$
then the subset $B$ is a \emph{sink} of $(A,I^\circ)$.
\end{enumerate}
\end{definition}

\begin{remark}\label{R:SbWaSink}
In Definition~\ref{D:SbWaSink}(a), the subalgebra $B$ is a barycentric algebra $(B,I^\circ)$ in its own right.
\end{remark}

\begin{example}
If $(A,I^\circ)$ is a polytope according to Example~\ref{X:CoAfBAlg}, then $B$ is a wall if and only if it is a \emph{face} in the terminology of \cite[\S\S5,7]{AB83}.
\end{example}

\begin{example}
In the barycentric algebra $(\mathbb R^\infty,I^\circ)$ of extended reals, as in Example~\ref{X:ExtReals}, the subset $\mathbb R$ is a subalgebra, and the subset $\set\infty$ is a sink.
\end{example}

\begin{lemma}
Walls and sinks are subalgebras.
\end{lemma}

\begin{proposition}
Let $(A,I^\circ)$ be a barycentric algebra.
\begin{enumerate}
\item[$(\mathrm a)$]
Suppose that, for some index set $G$, there is a family $\set{B_\gamma|\gamma\in G}$ of subalgebras, walls, or sinks of $(A,I^\circ)$. Then the intersection
$
\bigcap_{\gamma\in G}B_\gamma
$
of the family is again (respectively) a subalgebra, wall, or sink of $(A,I^\circ)$.
\item[$(\mathrm b)$]
The respective sets of subalgebras, walls, or sinks of $(A,I^\circ)$ form complete lattices.
\end{enumerate}
\end{proposition}

\begin{definition}
Let $(A,I^\circ)$ be a barycentric algebra, with a subset $S$. Then the \emph{subalgebra}, \emph{wall}, or \emph{sink} $\braket S$ \emph{generated} by $S$ is the intersection of all the subalgebras, walls, or sinks of $(A,I^\circ)$ that contain $S$.
\end{definition}

\begin{example}
Let $S$ be a set of points in a real affine space $A$. Interpret $A$ as a barycentric algebra $(A,I^\circ)$ according to Example~\ref{X:CoAfBAlg}. Then the \emph{convex hull} of $S$ is the subalgebra $\braket S$ of $(A,I^\circ)$ that is generated by $S$.
\end{example}

\subsubsection{Homomorphisms of barycentric algebras}

\begin{definition}\label{D:pontwise}
Suppose that $X$ and $Y$ are sets, and that $(A',I^\circ)$ is a barycentric algebra.
\begin{enumerate}
\item[$(\mathrm a)$]
Write $\mathbf{Set}(X,Y)$ for the set of all functions from $X$ to $Y$.
\item[$(\mathrm b)$]
For $p\in I^\circ$, define the \emph{pointwise} or \emph{componentwise} operation
\begin{equation}\label{E:pontwise}
fg\,\underline p\colon X\to A';x\mapsto x^fx^g\,\underline p
\end{equation}
on elements $f,g$ of $\mathbf{Set}(X,A')$.
\end{enumerate}
\end{definition}

\begin{proposition}\label{P:pontwise}
Within the context of Definition~\ref{D:pontwise}, the set $\mathbf{Set}(X,A')$ becomes a barycentric algebra under the pointwise operations \eqref{E:pontwise}.
\end{proposition}

\begin{proof}
The idempotence, skew-commutativity and skew-associativity for the pointwise operations on $\left(\mathbf{Set}(X,A'),I^\circ\right)$ arise as direct consequences of the corresponding properties of $(A',I^\circ)$.
\end{proof}

\begin{definition}\label{D:BaHomSet}
Suppose that $(A,I^\circ)$ and $(A',I^\circ)$ are barycentric algebras.
\begin{enumerate}
\item[$(\mathrm a)$]
A function $f\colon A\to A';x\mapsto x^f$ is said to be a \emph{barycentric} (\emph{algebra}) \emph{homomorphism} if
\begin{equation}\label{E:BaryAHom}
xy\,\underline p^f=x^fy^f\,\underline p
\end{equation}
for all $x,y\in A$ and $p\in I^\circ$.
\item[$(\mathrm b)$]
A homomorphism is an \emph{isomorphism} if it is bijective.
\item[$(\mathrm c)$]
Write $\mathbf B(A,A')$ for the set of all barycentric homomorphisms from $(A,I^\circ)$ to $(A',I^\circ)$.
\end{enumerate}
\end{definition}

\begin{proposition}\label{P:BaHomSet}
Within the context of Definition~\ref{D:BaHomSet}, the set $\mathbf B(A,A')$ becomes a barycentric algebra under the pointwise operations \eqref{E:pontwise}.
\end{proposition}

\begin{proof}
The computation
\begin{align*}
(xy\,\underline q)^{fg\,\underline p}
&
\overset{\eqref{E:pontwise}}=
(xy\,\underline q)^f(xy\,\underline q)^g\,\underline p
\overset{\eqref{E:BaryAHom}}=
(x^fy^f\,\underline q)(x^gy^g\,\underline q)\underline p
\\
&
\overset{\eqref{E:entropic}}=
(x^fx^g\,\underline p)(y^fy^g\,\underline p)\underline q
\overset{\eqref{E:pontwise}}=
x^{fg\,\underline p}y^{fg\,\underline p}\underline q
\end{align*}
for $x,y\in A$ and $q\in I^\circ$ shows that the function $fg\,\underline p\colon A\to A'$ of \eqref{E:pontwise} is a barycentric algebra homomorphism. Thus, the subset $\mathbf B(A,A')$ is a subalgebra of the barycentric algebra $\mathbf{Set}(A,A')$ of Proposition~\ref{P:pontwise}.
\end{proof}

\subsection{The category of barycentric algebras}\label{SS:CatBaryA}

Since the composite of two barycentric homomorphisms
$$
f\colon(A,I^\circ)\to(A',I^\circ);x\mapsto x^f
\mbox{ and  }
g\colon(A',I^\circ)\to(A'',I^\circ);y\mapsto y^g
$$
is a barycentric algebra homomorphism $fg\colon(A,I^\circ)\to(A'',I^\circ);x\mapsto x^{fg}$, the class $\mathbf B_0$ of barycentric algebras forms (the object class of) a category $\mathbf B$ of barycentric algebras, with the set $\mathbf B(A,A')$  of Definition~\ref{D:pontwise} as the set of morphisms from $A$ to $A'$.

There is a forgetful functor $U\colon\mathbf B\to\mathbf{Set}$ from $\mathbf B$ to the category $\mathbf{Set}$ of sets, defined by the subset relationship $\mathbf B(A,A')\subseteq\mathbf{Set}(A,A')$. This functor has a left adjoint
\begin{equation}\label{E:FBAlgFct}
B\colon\mathbf{Set}\to\mathbf B
\end{equation}
which assigns the \emph{free barycentric algebra} $XB$ to a set $X$, together with an insertion function (adjunction unit) $\eta_X\colon X\to XBU$. Thus, for each barycentric algebra $(A,I^\circ)$ and function $f\colon X\to A$, there is a unique barycentric homomorphism $\overline f\colon XB\to A$ such that $\eta_X\overline f=f$.

\subsubsection{Construction of the free barycentric algebra}\label{SSS:FreeBAlg}

Given a set $X$, and the real line $(\mathbb R,I^\circ)$ as a barycentric algebra according to Example~\ref{X:CoAfBAlg}, take the barycentric algebra $\mathbf{Set}(X,\mathbb R)$ of Proposition~\ref{P:pontwise}. Consider the function
\begin{equation}\label{E:delta2xR}
X\to\mathbf{Set}(X,\mathbb R);
x\mapsto
\big[
\,
\delta_x\colon y\mapsto
\mbox{ \textbf{if} } y=x \mbox{ \textbf{then} } 1 \mbox{ \textbf{else} } 0
\,
\big]
\,.
\end{equation}
Define $XB$ as the subalgebra $\braket{\set{\delta_x|x\in X}}$ of $\mathbf{Set}(X,\mathbb R)$ generated by the set of delta functions, and define the insertion function $\eta_X\colon X\to XBU$ to be the corestriction of \eqref{E:delta2xR} to the subset $XBU$ of $\mathbf{Set}(X,\mathbb R)$.

By the Well-Ordering Theorem, the set $X$ may be given a total order $(X,\le)$. Repeated application of the axioms \eqref{E:idemptnc}--\eqref{E:skewasoc} for a barycentric algebra then shows that each element of $XB$ has a unique expression of the form
\begin{equation}\label{E:deltasqs}
\delta_{x_0}\delta_{x_1}\dots\delta_{x_r}\,\underline q_1\dots\underline q_r
\end{equation}
for some $r\in\mathbb N$, ordered subset $\set{x_0<x_1\dots<x_r}$ of $X$, and operators $q_1,\dots, q_r\in I^\circ$ (compare \cite[Lemma~5.8.1]{Modes}). The image of \eqref{E:deltasqs} under the extension $\overline f$ is then taken as the barycentric combination $x_0^fx_1^f\dots x_r^f\,\underline q_1\dots\underline q_r$ in $(A,I^\circ)$.

\subsubsection{Probability distributions and weighted averages}\label{SSS:PrDiWeAv}

The element \eqref{E:deltasqs} may be written as the convex combination
\begin{equation}\label{E:deltasps}
\sum_{k=0}^r\delta_{x_k}p_k
\end{equation}
with coefficients $p_k=q_kq_{k+1}'\dots q_r'\in I^\circ$ for $0\le k\le r$ and $q_0=1$. Thus, barycentric combinations may be taken as finitely supported probability distributions or weighted averages. In \eqref{E:deltasps}, weight $p_k$ is attached to the representation $x_k\eta_X=\delta_{x_k}$ of $x_k$. The weights are barycentric coordinates in the sense of \cite[\S31]{Moebius}.

\subsubsection{Products of barycentric algebras}\label{SSS:ProdBAlg}

For barycentric algebras $(A,I^\circ)$ and $(A',I^\circ)$, their \emph{product} $A\times A'$ is defined as the set $\set{(x,x')|x\in A\,,\ x'\in A'}$ with componentwise structure
$
(x,x')(y,y')\underline p=\left(xy\,\underline p,x'y'\,\underline p\right)
$
and \emph{projections}
$$
\pi\colon A\times A'\to A;(x,x')\mapsto x\,,\quad
\pi'\colon A\times A'\to A;(x,x')\mapsto x'\,.
$$
There is then a natural isomorphism
$$
\mathbf B(A'',A\times A')\cong\mathbf B(A'',A)\times\mathbf B(A'',A')
$$
for each barycentric algebra $(A'',I^\circ)$ given by $f\mapsto(f\pi,f\pi')$. Note that the product as described here is often called the \emph{direct product} or \emph{Cartesian product}.

\subsubsection{Symmetric monoidal categories}

The category $\mathcal L$ of real vector spaces and linear transformations carries the additonal structure $(\mathcal L,\otimes,\mathbb R)$ of a symmetric monoidal category \cite[p.180]{MacLane}, with the one-dimensional real vector space $\mathbb R$, the ``free vector space on one generator'', as a unit element for the usual tensor product $U\otimes V$ of vector spaces $U$ and $V$. Since the set $\mathcal L(U,V)$ of linear transformations from $U$ to $V$ is itself a vector space, the (natural) isomorphism
\begin{equation}\label{E:LUVWLUVW}
\mathcal L(U,\mathcal L(V,W))\cong\mathcal L(U\otimes V,W)\,,
\end{equation}
interpreting bilinear functions $U\times V\to W$ to a vector space $W$ as linear functions $U\otimes V\to W$, serves to specify the tensor product.

The category $\mathbf B$ of barycentric algebras carries a similar structure \cite{DD85,Sm168}. By analogy with \eqref{E:LUVWLUVW}, the (natural) isomorphism
\begin{equation}\label{E:BAAABAAA}
\mathbf B(A,\mathbf B(A',A''))\cong\mathbf B(A\otimes A',A'')
\end{equation}
defines a \emph{tensor product} $A\otimes A'$ of barycentric algebras $A$ and $A'$, with the singleton barycentric algebra as a unit. Note the use in \eqref{E:BAAABAAA} of the fact (Proposition~\ref{P:BaHomSet}) that $\mathbf B(A',A'')$ is a barycentric algebra.

If $X$ and $Y$ are sets, then the isomorphism
\begin{equation}\label{E:XYB2XBYB}
(X\times Y)B\cong XB\otimes YB
\end{equation}
identifies the free barycentric algebra on the direct product $X\times Y$ as the tensor product of the respective free algebras on $X$ and $Y$. (Both sides of the isomorphism satisfy the same categorical universality property.) The isomorphism \eqref{E:XYB2XBYB} is the counterpart of the observation from linear algebra that the dimension of a tensor product of vector spaces is the product of the dimensions of the factors.

\subsection{A pocket dictionary}

When a particular topic is analyzed from the standpoint of different fields, it becomes necessary to reconcile the differing and sometimes conflicting terminologies used in those fields. The basic assumption for this section is that a barycentric algebra $A$ is generated by a finite set $V$. We attempt to correlate the diverging terminologies of algebra and analysis.

\subsubsection{Free algebras and simplices}\label{SSS:fralasmp}

As a convex set barycentric algebra, a finite-dimensional simplex $A$ is generated by its vertex set $V$. Simplices are standard models of finitely generated free barycentric algebras as considered in \S\ref{SSS:FreeBAlg}. The freeness means that each function $f\colon V\to C$ from the generating set $V$ to (the underlying set of) a barycentric algebra $C$ has a unique extension to a barycentric homomorphism $\overline f\colon A\to C$.

\subsubsection{Kernel functions}\label{SSS:KernlFun}

Under the assumption of this section, each element $a$ of $A$ may be written as a barycentric combination
\begin{equation}\label{E:brycmbkn}
a=\sum_{v\in V}p(a,v)v
\end{equation}
with some specific choice of $p(a,v)\in I$, as observed in \S\ref{SSS:PrDiWeAv}.\footnote{The full real unit interval $I$ is required here, since some of the weights may be zero, and one of the weights may actually be $1$.} Suppose that this has been done. One may then focus on the function
\begin{equation}\label{E:kernelfn}
p\colon A\times V\to I;(a,v)\mapsto p(a,v)
\end{equation}
which, in the language of analysis, may be called a \emph{kernel} (\emph{function}) due to its role in \eqref{E:brycmbkn}, or in continuous versions of \eqref{E:brycmbkn} such as \cite[\S3.1]{KosBarW}. Examples are given in Definitions~\ref{D:Gibalpha} and \ref{D:WxprsPlr}.

\subsubsection{Partitions of unity and linear precision}\label{SSS:P1LinPre}

The curried (or parametrized univariate) versions
\begin{equation}\label{E:CoordFun}
A\to I;a\mapsto p(a,v)
\quad
\mbox{ for each }
\
v\in V
\end{equation}
of the bivariate kernel function \eqref{E:kernelfn} are viewed as \emph{coordinate functions} (or ``barycentric coordinates'' \cite[\S1.1]{WarSchHirDes}). The equation
\begin{equation}\label{E:PartoOne}
\sum_{v\in V}p(a,v)=1
\end{equation}
is then expressed as saying that the coordinate functions form a \emph{partition of unity}, while \eqref{E:brycmbkn} is said to express the \emph{barycentric property} \cite[(3)]{KosBarW} or \emph{linear precision} \cite[\S1.1]{WarSchHirDes}. We record the following for subsequent reference.

\begin{lemma}\label{L:BPimpPO1}
If the barycentric property \eqref{E:brycmbkn} holds, then the partition of unity property \eqref{E:PartoOne} follows.
\end{lemma}

\begin{proof}
Suppose that \eqref{E:brycmbkn} holds for an element $a$ of $A$. Then, we may consider the constant barycentric homomorphism $k:A\to\set 1$. Thus \eqref{E:PartoOne} is obtained as the image of \eqref{E:brycmbkn} under $k$.
\end{proof}

\subsubsection{Interpolation}

A function $f\colon V\to C$ from $V$ to some codomain barycentric algebra $C$ may be \emph{extended} (in the algebraist's terminology) or \emph{interpolated} (in the analyst's terminology) to a function $\widehat f\colon A\to C$ with
$$
\widehat f(a)=\sum_{v\in V}p(a,v)f(v)
$$
making use of \eqref{E:brycmbkn} \cite[(2)]{WarSchHirDes}. As discussed in \S\ref{SSS:fralasmp}, the extension $\widehat f$ will be uniquely specified as the barycentric homomorphism $\overline f$ when  $A$ is freely generated by $V$.

\subsubsection{Direct product or tensor product?}\label{SSS:dipotepo}

Note that the direct product of barycentric algebras (\S\ref{SSS:ProdBAlg}) may be described as a ``tensor product'' in the literature that is focused on linearity, e.g:
\begin{quote}
``The unit square is the tensor of two unit intervals''
\end{quote}
\cite[p.98]{Warren}.
Since the unit interval implements the free barycentric algebra on two generators (the binary digits $0$ and $1$), \eqref{E:XYB2XBYB} shows that the tensor of two unit intervals, namely the free algebra on $\set{00,01,10,11}$, is actually a tetrahedron, not the unit square.

\section{Gibbs coordinates}\label{S:GibbsBAl}

\subsection{Canonical distributions}\label{SS:CanoDist}

We summarize the construction of Gibbs coordinates from \cite[Ch.~IX]{Modes}, in a notation that is adapted to the current discussion. Let $X$ be a finite set of \emph{microstates}. Suppose that $A$ is a barycentric algebra that is generated by the image of a \emph{valuation function} $f\colon X\to A;x\mapsto x^f$. The generation of the algebra is not required to be irredundant. For example, if $A$ is a polytope, then it is not required that all the elements of the image be extreme. Since $A$ is generated by the image $Xf$, each element $\alpha$ of $A$ may be written as a barycentric combination
\begin{equation}\label{E:alphapax}
\alpha=\sum_{x\in X}p_xx^f
\end{equation}
with a probability distribution $p_x$ on $X$ (cf. \S\ref{SSS:PrDiWeAv}). Note that $p_x$ need not necessarily be the unique distribution for which \eqref{E:alphapax} holds. For $\alpha$ in $A$, the quantity
\begin{equation}\label{E:entalpha}
H(\alpha):=\sup\set{-\sum_{x\in X}p_x\log p_x
|\alpha=\sum_{x\in X}p_xx^f}
\end{equation}
is defined as the \emph{entropy} of the element $\alpha$. Here, $0\log 0=0$.

Consider the barycentric algebra $\mathbb R^\infty$ of extended real numbers presented in Example~\ref{X:ExtReals}. Define $e^{-\infty}=0$. Define $A^*$ to be the set $\mathbf B(A,\mathbb R^\infty)$ of all barycentric homomorphisms $\beta$ from $A$ to $\mathbb R^\infty$. By Proposition~\ref{P:BaHomSet}, it forms a barycentric algebra under the pointwise structure inherited from $\mathbb R^\infty$. Now consider an element $\beta$ of $A^*$. As a barycentric homomorphism, it is already specified by the function
$$
X\to\mathbb R^\infty;x\mapsto xf\beta\,.
$$
Define the \emph{partition function}
\begin{equation}\label{E:PartyFun}
Z\colon A^*\to[0,\infty[;\beta\mapsto\sum_{x\in X}e^{-xf\beta} \,.
\end{equation}
Then an element $\beta$ of $A^*$ determines the \emph{Gibbs} or \emph{canonical} probability distribution
\begin{equation}\label{E:GibsDist}
q^\beta_x=\frac{e^{-xf\beta}}{Z(\beta)}
\end{equation}
on $X$.

\begin{theorem}\label{T:Gibalpha}\cite[Th.~IX.9.8.1]{Modes}
Consider an element $\alpha$ of $A$.
\begin{enumerate}
\item[$(\mathrm a)$]
There is an element $\beta$ of $A^*$ such that
\begin{equation}\label{E:alphabet}
\alpha=\sum_{x\in X}q^\beta_xx^f \,.
\end{equation}
\item[$(\mathrm b)$]
The equality
$$
H(\alpha)=-\sum_{x\in X}q^\beta_x\log q^\beta_x
$$
holds.
\item[$(\mathrm c)$]
Amongst all probability distributions $p_x$ on $X$ for which \eqref{E:alphapax} holds, the Gibbs distribution $q^\beta_x$ is the unique one for which the supremum $H(\alpha)$ of \eqref{E:entalpha} is attained.
\end{enumerate}
\end{theorem}

\begin{example}
Consider the insertion $j\colon\set{0,1}\hookrightarrow[0,1]$ of the endpoints into the closed unit interval. Consider the potentials
$
\beta\colon 0\mapsto a\,,\ 1\mapsto b
$
for $a,b\in\mathbb R$. Then as instances of \eqref{E:alphabet}, we have
$$
\sum_{x\in\set{0,1}}q^\beta_xx^j=\frac{e^{-b}}{e^{-a}+e^{-b}}\,.
$$
If $a=0\ll b$, we get close to $0$. If $a$ and $b$ are comparable, we are somewhere in the middle. If $b=0\ll a$, we get close to $1$. For the endpoints of the interval, see Remark~\ref{R:GibsDu2A}.
\end{example}

\begin{definition}\label{D:Gibalpha}
Within the context of Theorem~\ref{T:Gibalpha}, the non-negative real numbers $q^\beta_x$ are known as the \emph{Gibbs coordinates} of the element $\alpha$ of $A$. We then set $q^\beta_x=q^\beta(\alpha,x)$, following \eqref{E:kernelfn}.
\end{definition}

The relation \eqref{E:alphabet} may be rewritten as
$$
\alpha=\sum_{x\in X}q^\beta(\alpha,x)x^f
$$
using the notation of Definition~\ref{D:Gibalpha} when we wish to emphasize the role of $q^\beta(\alpha,x)$ as a kernel function.

\subsection{Potentials and gauge transformations}

Appearing in the context of Theorem~\ref{T:Gibalpha}, the elements $\beta$ of $A^*$ are known as \emph{potentials}. Although the Gibbs distribution $q^\beta(\alpha,x)$ is uniquely determined by $\alpha$, the potential $\beta$ is not.

For the following proposition, extend the addition $+$ of real numbers to a binary operation on $\mathbb R^\infty$ by declaring $r+\infty=\infty=\infty+r$ for all $r\in\mathbb R^\infty$.

\begin{proposition}\label{P:GaugeTrs}
Let $k$ be a real number. Suppose that $A$ is a barycentric algebra generated by the image of a valuation function $f\colon X\to A$ with finite domain. Let $\beta$ be an element of $A^*$.
\begin{enumerate}
\item[$(\mathrm a)$]
The function
$$
\beta+k\colon A\to\mathbb R^\infty;\alpha\mapsto\alpha\beta+k
$$
is a barycentric homomorphism.
\item[$(\mathrm b)$]
Under
$$
k\colon\beta\mapsto\beta+k\,,
$$
the additive group $(\mathbb R,+,0)$ of real numbers acts on $A^*$.
\item[$(\mathrm c)$]
The Gibbs distributions $q^\beta_x$ and $q^{\beta+k}_x$ on $X$ coincide for each $\beta\in A^*$ and $k\in\mathbb R$.
\end{enumerate}
\end{proposition}

\begin{proof}
(a)
For $p\in I^\circ$ and $\alpha_i\in A$, the relation $(\alpha_1p'+\alpha_2p)(\beta+k)=
\alpha_1\beta p'+\alpha_2\beta p+kp'+kp
=\alpha_1(\beta+k)p'+\alpha_2(\beta+k)p
$
holds, as required.
\vskip 2mm
\noindent
(b)
Note
$
\alpha\left((\beta+k_1)+k_2\right)
=\alpha\beta+k_1+k_2
=\alpha\left(\beta+(k_1+k_2)\right)
$
for $\alpha\in A$, $\beta\in A^*$, and $k_i\in\mathbb R$.
\vskip 2mm
\noindent
(c)
Note $Z(\beta+k)=e^{-k}Z(\beta)$, and $e^{-xf(\beta+k)}=e^{-k}e^{-xf\beta}$ for $x\in X$. Thus, comparing the expressions \eqref{E:GibsDist} for $q^{\beta+k}_x$ and $q^{\beta}_x$, the extra factor of $e^{-k}$ in the numerator and denominator for $q^{\beta+k}_x$ cancels out, reducing to $q^{\beta}_x$.
\end{proof}

\begin{definition}\label{D:GaugeTrs}
In the context of Proposition~\ref{P:GaugeTrs}(b), the actions of real numbers are described as \emph{gauge transformations}, and the additive group $(\mathbb R,+,0)$ is called the \emph{gauge group}.
\end{definition}

\begin{proposition}\label{P:GaGrOrCs}
For each extended real number $r\in\mathbb R^\infty$, write $\beta_r$ for the constant homomorphism in $\mathbf B(A,\mathbb R^\infty)$ with image $\set r$. Then
\begin{equation}\label{E:cstggorb}
\set{\beta_\infty}
\
\mbox{ and }
\
\set{\beta_k|k\in\mathbb R}
\end{equation}
are the orbits of the gauge group $\mathbb R$ on the set $\set{\beta_r|r\in\mathbb R^\infty}$ of constant potentials.
\end{proposition}

\begin{remark}
The interpretation of $q_x^\infty$ as the uniform distribution on $X$ is discussed in \cite[Rem.~9.6.8]{Modes}.
\end{remark}

\subsection{The Gibbs dual}

Let $A$ be a barycentric algebra.

\begin{definition}
The \emph{Gibbs dual} of $A$ is
the set
$$
A^*_{/\mathbb R}:=\set{\beta+\mathbb R|\beta\in A^*}
$$
of orbits of $A^*$ under the action of the gauge group $(\mathbb R,+,0)$.
\end{definition}

\begin{proposition}
The surjective map
$$
\varepsilon_{/\mathbb R}\colon A^*\to A^*_{/\mathbb R};
\beta\mapsto\beta+\mathbb R
$$
is a well-defined barycentric homomorphism.
\end{proposition}

\begin{proof}
Suppose that $\beta'_i=\beta_i+k_i$ for $i=1,2$, with $\beta_i,\beta'_i\in A^*$ and $k_i\in\mathbb R$. Then for each element $\alpha\in A$ and $p\in I^\circ$, we have
\begin{align*}
\alpha(\beta'_1\beta'_2\underline p)
&
=\alpha\beta'_1p'+\alpha\beta'_2p
=\alpha(\beta_1+k_1)p'+\alpha(\beta_2+k_2)p
\\
&
=\alpha\beta_1p'+k_1p'+\alpha\beta_2p+k_2p
=\alpha\beta_1p'+\alpha\beta_2p+k_1p'+k_2p
\\
&
=\alpha(\beta_1\beta_2\underline p)+k_1k_2\underline p\,,
\end{align*}
so that $\beta'_1\beta'_2\underline p=\beta_1\beta_2\underline p+k_1k_2\underline p$
as required.
\end{proof}

Now suppose that $A$ is generated by the image of a valuation function $f\colon X\to A$ from a finite set $X$ of microstates.

\begin{proposition}\label{P:GibsDu2A}
There is a well-defined map
\begin{equation}\label{E:exA*ftoA}
q_f\colon A^*_{/\mathbb R}\to A;\beta+\mathbb R\mapsto\sum_{x\in X}q^\beta_xx^f
\end{equation}
from the Gibbs dual of $A$ to $A$.
\end{proposition}

\begin{proof}
The well-definition of \eqref{E:exA*ftoA} follows from Proposition~\ref{P:GaugeTrs}(c).
\end{proof}

\begin{remark}\label{R:GibsDu2A}
(a)
In general, \eqref{E:exA*ftoA} is not a barycentric homomorphism. For example, consider the insertion $j\colon\set{0,1}\hookrightarrow[0,1]$ of the endpoints into the closed unit interval. Consider the potentials
\begin{align*}
\beta\colon0\mapsto0\,,\ ]0,1]\to\set\infty
\
\mbox{ and }
\
\beta'\colon[0,1[\to\set\infty\,,\ 1\mapsto0
\end{align*}
with
\begin{align*}
q_j\colon(\beta+\mathbb R)\mapsto 0\,,
\quad
q_j\colon(\beta'+\mathbb R)\mapsto1
\
\mbox{ and }
\
q_j\colon(\beta\beta'\underline p+\mathbb R)\mapsto\tfrac12
\end{align*}
for $p\in I^\circ$, since $\beta\beta'\underline p=\beta_\infty$. (Compare \cite[Def'n.~9.6.1]{Modes} for the relevant convention.) If $p\ne\frac12$, we have
$$
(\beta+\mathbb R)^{q_j}(\beta'+\mathbb R)^{q_j}\underline p
=01\underline p=p
\ne\tfrac12=
(\beta\beta'\underline p+\mathbb R)^{q_j} \,.
$$
\vskip 2mm
\noindent
(b)
Note that the two elements \eqref{E:cstggorb} of the Gibbs dual have the same image under \eqref{E:exA*ftoA}. (Again, compare \cite[Def'n.~9.6.1]{Modes} for the relevant convention.)
\end{remark}

\section{Wachspress coordinates}\label{S:Wchspres}

\subsection{Volumetric coordinates in a simplex}\label{SS:volmtric}

The (unique) barycentric coordinates in a simplex may be given in terms of relative (hyper)volumes. These barycentric coordinates extend to affine coordinates for the entire affine space containing the simplex.

Consider a simplex of geometric dimension $n$ that is spanned by the ordered set $\mathbf v_0,\dots,\mathbf v_{i-1},\mathbf v_i,\mathbf v_{i+1},\dots,\mathbf v_n$ in $n$-dimensional Euclidean space. Here, the cyclic order is uniquely determined by topological orientation. The signed (hyper)volume of the simplex is
\begin{equation}\label{E:hypvolv0}
\tfrac1{n!}
\det[
\mathbf v_1-\mathbf v_0,
\dots,
\mathbf v_{i-1}-\mathbf v_0,
\mathbf v_i-\mathbf v_0,
\mathbf v_{i+1}-\mathbf v_0,
\dots,
\mathbf v_n-\mathbf v_0
]
\end{equation}
or
\begin{equation}\label{E:hypvhomg}
\tfrac1{n!}
\det[
1\oplus\mathbf v_0,
1\oplus\mathbf v_1,
\dots,
1\oplus\mathbf v_{i-1},
1\oplus\mathbf v_i,
1\oplus\mathbf v_{i+1},
\dots,
1\oplus\mathbf v_n
]
\end{equation}
in a more symmetric format that is related to homogeneous coordinates in projective geometry. (For the case $n=2$, see the subsequent section.)

Now, for a point $\mathbf x$ of the simplex, define
\begin{equation}\label{E:volmtrpi}
p_i=\frac
{
\det[
\mathbf v_1-\mathbf v_0,
\dots,
\mathbf v_{i-1}-\mathbf v_0,
\mathbf x-\mathbf v_0,
\mathbf v_{i+1}-\mathbf v_0,
\dots,
\mathbf v_n-\mathbf v_0
]
}
{
\det[
\mathbf v_1-\mathbf v_0,
\dots,
\mathbf v_{i-1}-\mathbf v_0,
\mathbf v_i-\mathbf v_0,
\mathbf v_{i+1}-\mathbf v_0,
\dots,
\mathbf v_n-\mathbf v_0
]
}
\,.
\end{equation}
Then for $i>0$, the $p_i$ are the (unique) barycentric coordinates of $\mathbf x$ with respect to $\mathbf v_i$ in the simplex, so that
\begin{equation}\label{E:volmtric}
\mathbf x=\sum_{i=0}^np_i\mathbf v_i
\end{equation}
with
$$
p_0=1-\sum_{i=1}^np_i\,.
$$
The convex combination \eqref{E:volmtric} based on \eqref{E:volmtrpi} is said to give the \emph{volumetric} coordinates for $\mathbf x$ in terms of the vertices $\mathbf v_i$.

The formula \eqref{E:volmtric} extends seamlessly to points $\mathbf x$ that are outside the simplex. Here, certain of the numerators in \eqref{E:volmtrpi} will become negative, and the expression \eqref{E:volmtric} expresses $\mathbf x$ as an affine combination of the vertices $\mathbf v_i$ that is not a convex combination.

\subsubsection{Areal coordinates}\label{SSS:Aov0v1v2}

The volumetric coordinates specialize, in the case of a triangle, to the classical \emph{areal coordinates} \cite{Muggeridge}. Here, it is convenient to set
\begin{equation}\label{E:Aov0v1v2}
A\left(\mathbf v_0,\mathbf v_1,\mathbf v_2\right)
=
\tfrac12\det[
\mathbf v_1-\mathbf v_0,
\mathbf v_2-\mathbf v_0
]
=
\frac12
\begin{vmatrix}
1 &v_{00} &v_{01}\\
1 &v_{10} &v_{11}\\
1 &v_{20} &v_{21}\\
\end{vmatrix}
\end{equation}
for the signed area of a triangle in the Euclidean plane whose vertices are
$
\mathbf v_0=
\begin{bmatrix}
v_{00} &v_{01}
\end{bmatrix},
\mathbf v_1=
\begin{bmatrix}
v_{10} &v_{11}
\end{bmatrix}
\mbox{ and }
\mathbf v_2=
\begin{bmatrix}
v_{20} &v_{21}
\end{bmatrix}
$
in counterclockwise order,
specializing \eqref{E:hypvolv0} and \eqref{E:hypvhomg} to obtain the respective equations in \eqref{E:Aov0v1v2}.

The following proposition recalls the convexity of the area $A\left(\mathbf v_0,\mathbf v_1,\mathbf v_2\right)$ with respect to variation of the vertex $\mathbf v_0$.

\begin{proposition}\label{P:Aov0v1v2}
Consider vectors
$$
\mathbf v_0=
\begin{bmatrix}
v_{00} &v_{01}
\end{bmatrix},
\mathbf v_0'=
\begin{bmatrix}
v_{00}' &v_{01}'
\end{bmatrix},
\mathbf v_1=
\begin{bmatrix}
v_{10} &v_{11}
\end{bmatrix}
\mbox{ and }
\mathbf v_2=
\begin{bmatrix}
v_{20} &v_{21}
\end{bmatrix}
$$
in the plane. Then
$$
A\big((1-p)\mathbf v_0+p\mathbf v_0',\mathbf v_1,\mathbf v_2\big)
=
(1-p)
A\left(\mathbf v_0,\mathbf v_1,\mathbf v_2\right)
+
p
A\left(\mathbf v_0',\mathbf v_1,\mathbf v_2\right)
$$
for each real number $p$.
\end{proposition}

\begin{proof}
The result follows by consideration of the Laplace expansion
$$
A\left(\mathbf v_0,\mathbf v_1,\mathbf v_2\right)=
\frac12
\left(
\begin{vmatrix}
v_{10} &v_{11}\\
v_{20} &v_{21}
\end{vmatrix}
-v_{00}
\begin{vmatrix}
1 &v_{11}\\
1 &v_{21}
\end{vmatrix}
+v_{01}
\begin{vmatrix}
1 &v_{10}\\
1 &v_{20}
\end{vmatrix}
\right)
$$
of the determinant of \eqref{E:Aov0v1v2} across its first row.
\end{proof}

\subsection{Planar Wachspress coordinates}\label{SS:WxprsPlr}

Throughout this section, we will consider a convex polygon $\Pi$ presented as the convex hull of the counter-clockwise ordered sequence $\mathbf v_1,\dots,\mathbf v_n$ of extreme points located around its boundary.

\subsubsection{Planar Wachspress coordinates inside a polygon}

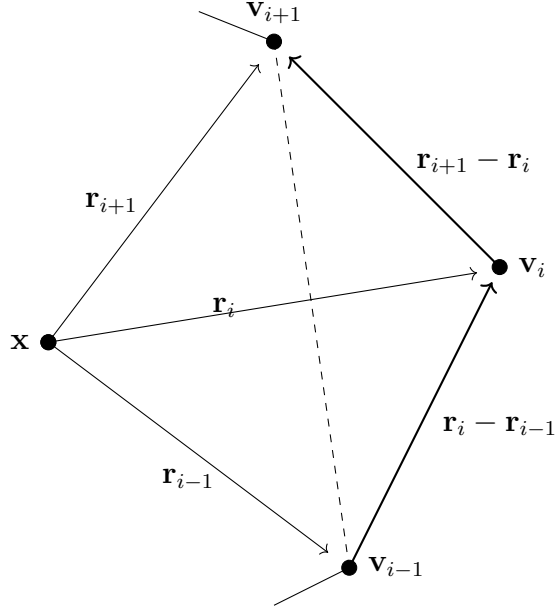
\begin{figure}[hbt]
  \centering
  \begin{tikzpicture}
    \coordinate (Origin)   at (0,0);
    \coordinate (XAxisMin) at (-5,0);
    \coordinate (XAxisMax) at (5,0);
    \coordinate (YAxisMin) at (0,-5);
    \coordinate (YAxisMax) at (0,5);
        \clip (-5,-5) rectangle (5.1cm,5.1cm);
\node[label=right:{$\mathbf v_{i-1}$}, draw, circle, inner sep=2pt, fill] at (1,-4) {};
\node[label=right:{$\mathbf v_{i}$}, draw, circle, inner sep=2pt, fill] at (3,0) {};
\node[label=above:{$\mathbf v_{i+1}$}, draw, circle, inner sep=2pt, fill] at (0,3) {};
\node[label=left:{$\mathbf x$},draw, circle, inner sep=2pt, fill] at (-3,-1){};
\node[draw, circle, inner sep=2pt, fill] at (-3,-1){};
\draw[thick,<-] (2.9,-0.2) -- node[right=4pt]
{$\mathbf{r}_{i}-\mathbf{r}_{i-1}$} (1,-4);
\draw[thick,<-] (0.2,2.8) -- node[right=4pt]
{$\mathbf{r}_{i+1}-\mathbf{r}_{i}$} (3,0);
\draw[thin] (0,-4.5) -- (1, -4);
\draw[->] (-3,-1) -- node[below=4pt] {$\mathbf{r}_{i-1}$} (.72, -3.8);
\draw[->] (-3,-1) -- node[above=4pt,left=6pt] {$\mathbf{r}_{i}$} (2.7,-0.075);
\draw[->] (-3,-1) -- node[left=2pt] {$\mathbf{r}_{i+1}$} (-.2,2.7);
\draw[thin] (0,3) -- (-1,3.4);
\draw[dashed] (1,-4) -- (0,3);
\end{tikzpicture}
\caption{Vectors for planar Wachspress coordinates (cf. \cite[Figure~3]{FloaterWMVC}).}
\label{F:WxprsPlr}
\end{figure}

Consider an interior point $\mathbf x$ of $\Pi$, as depicted in the local fragment of $\Pi$ appearing in Figure~\ref{F:WxprsPlr}. For $1\le i\le n$, define $\mathbf r_i=\mathbf v_i-\mathbf x$, which may be regarded as specifying the extreme points $\mathbf v_i$ with respect to the interior point $\mathbf x$. We refer to the vectors $\mathbf r_i$ as being taken in the \emph{radial} coordinate system. Thus the interior point $\mathbf x$ itself, as the origin of the radial system, has radial vector $\mathbf 0$. The following result presents a simplified version of the discussion provided in \cite[Section~3]{FloaterWMVC} on the basis of \cite{MeBaLeDe}.

\begin{proposition}\label{P:WxprsPlr}
For $1\le i\le n$, define the \emph{Wachspress weights}
\begin{align}
w_i(\mathbf x)
&
=
\frac
{
A\left(
\mathbf v_{i-1},\mathbf v_{i},\mathbf v_{i+1}
\right)
}
{
A\left(
\mathbf v_{i-1},\mathbf v_{i},\mathbf x
\right)
\cdot
A\left(
\mathbf x,\mathbf v_{i},\mathbf v_{i+1}
\right)
}
\label{E:WxprsPlr}
\end{align}
using the signed area notation \eqref{E:Aov0v1v2}. Then the interior point $\mathbf x$ of the polygon $\Pi$ may be expressed as the convex combination
\begin{equation}\label{E:WaConCom}
\mathbf x=\left(\sum_{j=1}^nw_j(\mathbf x)\right)^{-1}\sum_{i=1}^n\mathbf v_iw_i(\mathbf x)
\end{equation}
of the extreme points $\mathbf v_i$ of $\Pi$.
\end{proposition}

\begin{proof}
First, note that \eqref{E:WxprsPlr} is well defined. Indeed,
$
A\left(
\mathbf v_{i-1},\mathbf v_{i},\mathbf x
\right)
=0
$
would imply the linear dependence of $\set{\mathbf v_{i-1},\mathbf v_{i},\mathbf x}$, placing $\mathbf x$ on the edge joining $\mathbf v_{i-1}$ to $\mathbf v_{i}$.

Now, in the situation of Figure~\ref{F:WxprsPlr}, the areal coordinatization \eqref{E:volmtric} of $\mathbf x$ with respect to the simplex $\Delta_i$ spanned by $\mathbf v_{i-1},\mathbf v_{i},\mathbf v_{i+1}$ becomes
\begin{equation}\label{E:AreaCord}
\mathbf x
=\sum_{j=i-1}^{i+1}\mathbf v_j
\frac
{
A\left(
\mathbf v_{j-1},\mathbf x,\mathbf v_{j+1}
\right)
}
{
A\left(
\mathbf v_{i-1},\mathbf v_{i},\mathbf v_{i+1}
\right)
}
\end{equation}
taking suffix addition modulo $3$ here. Expressing this affine combination in radial coordinates, and multiplying by the simplex area, we have
$$
\sum_{j=i-1}^{i+1}\mathbf r_j
A\left(
\mathbf v_{j-1},\mathbf x,\mathbf v_{j+1}
\right)
=\mathbf 0
$$
or
\begin{equation}\label{E:riBiAi-A}
\mathbf r_iB_i
=A_i\left(\mathbf r_i-\mathbf r_{i-1}\right)
-A_{i-1}\left(\mathbf r_{i-1}-\mathbf r_i\right)
\end{equation}
on setting
\begin{equation}\label{BiAvAi-1}
B_i=A\left(
\mathbf v_{i-1},\mathbf v_{i},\mathbf v_{i+1}
\right)
\
\mbox{ and }
\
A_{i-1}=A\left(
\mathbf x,\mathbf v_{i-1},\mathbf v_{i}
\right)
\end{equation}
for $1\le i\le n$. Note $B_i=A_{i-1}+A_i+A\left(\mathbf v_{i-1},\mathbf x,\mathbf v_{i+1}\right)$ from Figure~\ref{F:WxprsPlr}. (For the latter equation, it may help to envisage the case of the figure with $\mathbf x$ inside $\Delta_i$.)

Having completed the local consideration at the simplex $\Delta_i$ depicted in Figure~\ref{F:WxprsPlr}, we zoom out to consider the entire polygon $\Pi$ spanned by the vertices $\mathbf v_i$ for $1\le i\le n$. Now, addition of the suffices will be interpreted modulo $n$, so that $A_n=A\left(
\mathbf x,\mathbf v_{n},\mathbf v_{1}
\right)$, for example. The Wachspress weight \eqref{E:WxprsPlr} may be rewritten as
$
w_i(\mathbf x)=
{B_i}/\left({A_{i-1}A_i}\right)\,.
$
Division of the equation \eqref{E:riBiAi-A} by $A_{i-1}A_i$ yields
$$
\mathbf r_iw_i(\mathbf x)
=\frac1{A_{i-1}}\left(\mathbf r_i-\mathbf r_{i-1}\right)
-\frac1{A_{i}}\left(\mathbf r_{i-1}-\mathbf r_i\right) \,.
$$
Addition of all these equations for $1\le i\le n$ leads to
$
\sum_{i=1}^n\mathbf r_iw_i(\mathbf x)=\mathbf 0
$
or
\begin{equation*}\label{E:WaCnvCmb}
\mathbf x\sum_{i=1}^nw_i(\mathbf x)=\sum_{i=1}^n\mathbf v_iw_i(\mathbf x)\,,
\end{equation*}
from which \eqref{E:WaConCom} follows immediately.
\end{proof}

\begin{definition}\label{D:WxprsPlr}
At an interior point $\mathbf x$ of the polygon $\Pi$, the coordinates
\begin{equation}\label{E:WaWtWaCo}
w(\mathbf x,\mathbf v_i)=
\frac
{w_i(\mathbf x)}
{\sum_{j=1}^nw_j(\mathbf x)}
\end{equation}
for $1\le i\le n$ are known as the \emph{Wachspress coordinates} of $\mathbf x$.
\end{definition}

\begin{remark}\label{R:WxprsPlr}
Note that the definition \eqref{E:WaWtWaCo} of the Wachspress coordinates is invariant under uniform scalar multiplication of the Wachspress weights. For this reason, Wachspress weight specifications like \eqref{E:WxprsPlr} are often replaced by scalar multiples, as in \S\ref{SSS:WchspsBd} below.

The rescaling of Wachspress weights is analogous to the action of the gauge transformations of Definition~\ref{D:GaugeTrs} on Gibbs coordinates.
\end{remark}

\subsubsection{Wachspress coordinates as a bulk-boundary correspondence}\label{SSS:BulkBdry}

We have transformed
$$
\begin{matrix}
\boxed{
\mbox{ radial coordinates }
\mathbf r_i
\mbox{ of the }
\mathbf v_i
\mbox{ with respect to }
\mathbf x\
}\\
\mbox{ into }\rule[-3.5mm]{0mm}{10mm}\\
\boxed{
\mbox{ Wachspress coordinates }
w(\mathbf x,\mathbf v_i)
\mbox{ of }
\mathbf x
\mbox{ with respect to the }
\mathbf v_i\,.\
}
\end{matrix}
$$
This transformation may be seen as a ``bulk-boundary correspondence''. The first box refers the extreme points $\mathbf v_i$ on the boundary to a point $\mathbf x$ in the bulk, while the second refers the point $\mathbf x$ in the bulk to the extreme points $\mathbf v_i$ on the boundary.

\subsubsection{Wachspress weights, support functions, discrete curvature}\label{SSS:WaWeSuFu}

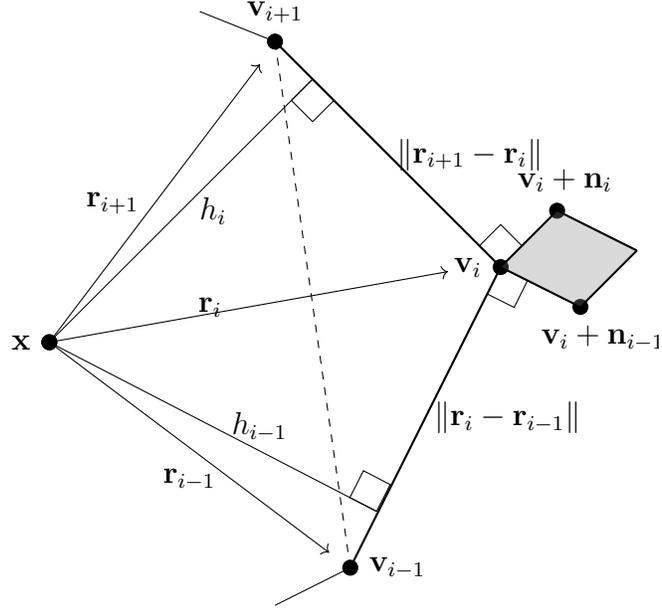
\begin{figure}[hbt]
  \centering
  \begin{tikzpicture}
    \coordinate (Origin)   at (0,0);
    \coordinate (XAxisMin) at (-5,0);
    \coordinate (XAxisMax) at (5,0);
    \coordinate (YAxisMin) at (0,-5);
    \coordinate (YAxisMax) at (0,5);
        \clip (-5,-5) rectangle (5.1cm,5.1cm);
\node[label=right:{$\mathbf v_{i-1}$}, draw, circle, inner sep=2pt, fill] at (1,-4) {};
\node[label=left:{$\mathbf v_{i}$}, draw, circle, inner sep=2pt, fill] at (3,0) {};
\node[label=above:{$\mathbf v_{i+1}$}, draw, circle, inner sep=2pt, fill] at (0,3) {};
\node[label=left:{$\mathbf x$},draw, circle, inner sep=2pt, fill] at (-3,-1){};
\node[draw, circle, inner sep=2pt, fill] at (-3,-1){};
\node[label=below:{$\rule{6mm}{0mm}\mathbf v_i+\mathbf n_{i-1}$},draw, circle, inner sep=2pt, fill] at (4.06,-0.53){};
\node[label=above:{$\rule{2mm}{0mm}\mathbf v_i+\mathbf n_{i}$},draw, circle, inner sep=2pt, fill] at (3.75,0.75){};
\draw[thick] (3,0) -- node[below=18pt, right=-1pt]
{$\|\mathbf{r}_{i}-\mathbf{r}_{i-1}\|$} (1,-4);
\draw[thick] (3,0) -- (3.75,0.75);
\draw[thick] (3,0) -- (4.06,-0.53);
\draw[thick] (3.75,0.75) -- (4.81,0.22);
\draw[thick] (4.06,-0.53) -- (4.81,0.22);
\filldraw[fill=gray, fill opacity=0.3, draw=black] (3,0) -- (3.75,0.75) --  (4.81,0.22) -- (4.06,-0.53) -- cycle;
\coordinate (V-1) at (1,-4) {};
\coordinate (V) at (3,0) {};
\coordinate (N-1) at (4.06,-0.53) {};
\coordinate (V+1) at (0,3) {};
\coordinate (N+1) at (3.75,0.75) {};
\coordinate (X) at (-3,-1) {};
\coordinate (X+1) at (0.5,2.5) {};
\coordinate (X-1) at (1.35,-3.25) {};
\tkzMarkRightAngle[size=.4](V-1,V,N-1);
\tkzMarkRightAngle[size=.4](V+1,V,N+1);
\tkzMarkRightAngle[size=.4](X,X+1,V);
\tkzMarkRightAngle[size=.4](X,X-1,V);
\draw[thin] (-3,-1) -- node[below=4pt,right=3pt]
{$h_i$} (0.5,2.5);
\draw[thin] (-3,-1) -- node[above=3pt,right=3pt]
{$h_{i-1}$} (1.35,-3.25);
\draw[thick] (0,3) -- node[above=23pt,right=-1pt]
{$\|\mathbf{r}_{i+1}-\mathbf{r}_{i}\|$} (3,0);
\draw[thin] (0,-4.5) -- (1, -4);
\draw[->] (-3,-1) -- node[below=4pt] {$\mathbf{r}_{i-1}$} (.72, -3.8);
\draw[->] (-3,-1) -- node[below=12pt,left=6pt] {$\mathbf{r}_{i}$} (2.3,-0.06);
\draw[->] (-3,-1) -- node[left=2pt] {$\mathbf{r}_{i+1}$} (-.2,2.7);
\draw[thin] (0,3) -- (-1,3.4);
\draw[dashed] (1,-4) -- (0,3);
\end{tikzpicture}
\caption{Planar Wachspress coordinates with curvature (cf. \cite[Figure~1]{WarSchHirDes}).}
\label{F:WxPlrCrv}
\end{figure}

Recall the notations
\begin{equation*}
B_i
=A\left(
\mathbf v_{i-1},\mathbf v_{i},\mathbf v_{i+1}
\right)
\
\mbox{ and }
\
A_{i-1}=A\left(
\mathbf x,\mathbf v_{i-1},\mathbf v_{i}
\right)
\end{equation*}
from \eqref{BiAvAi-1}. In terms of the lengths exhibited in Figure~\ref{F:WxPlrCrv}, we have
$$
A_i=\tfrac12 h_i\|\mathbf{r}_{i+1}-\mathbf{r}_{i}\| \,.
$$
In terms of the vectors exhibited in Figure~\ref{F:WxPlrCrv}, including the unit-length normal $\mathbf n_i$ to $\mathbf{r}_{i+1}-\mathbf{r}_{i}$ leading out of the polygon $\Pi$, we have
\begin{align*}
B_i
&
=
A\left(
\mathbf v_{i},\mathbf v_{i+1},\mathbf v_{i-1}
\right)
=
\tfrac12
\|(\mathbf{r}_{i+1}-\mathbf{r}_{i})\times(\mathbf{r}_{i-1}-\mathbf{r}_{i})\|
\\
&
=
\tfrac12
\|\mathbf{r}_{i+1}-\mathbf{r}_{i}\|
\cdot\|\mathbf{r}_{i}-\mathbf{r}_{i-1}\|
\cdot\|\mathbf n_i\times(-\mathbf n_{i-1})\|
\end{align*}
noting that $\mathbf n_i$ is rotated by $-\frac\pi2$ from $\mathbf{r}_{i+1}-\mathbf{r}_{i}$, and $-\mathbf n_{i-1}$ is rotated by $-\frac\pi2$ from $\mathbf{r}_{i-1}-\mathbf{r}_{i}$. We reformulate the Wachspress weight \eqref{E:WxprsPlr} as
\begin{align}\notag
w_i(\mathbf x)=
\frac{B_i}{A_{i-1}A_i}
&
=\frac
{2
\|\mathbf{r}_{i+1}-\mathbf{r}_{i}\|
\cdot\|\mathbf{r}_{i}-\mathbf{r}_{i-1}\|
\cdot\|\mathbf n_{i-1}\times\mathbf n_{i}\|
}
{
h_{i-1}\|\mathbf{r}_{i}-\mathbf{r}_{i-1}\|
\cdot
h_i\|\mathbf{r}_{i+1}-\mathbf{r}_{i}\|
}
\\ \label{E:wihihi-1}
&
=2\cdot
\frac
{
\det[\mathbf n_{i-1},\mathbf n_{i}]
}
{
h_{i-1}
h_i
}
\\ \label{E:CRVFRMwi}
&
=
\frac
{
2!\cdot
\det[\mathbf n_{i-1},\mathbf n_{i}]
}
{
\braket{\mathbf v_i-\mathbf x|\mathbf n_{i-1}}
\braket{\mathbf v_i-\mathbf x|\mathbf n_i}
}
\end{align}
writing $\braket{\mathbf v_i-\mathbf x|\mathbf n_{i-1}},\braket{\mathbf v_i-\mathbf x|\mathbf n_i}$ for the inner products of $\mathbf r_i$ with $\mathbf n_{i-1}$ and $\mathbf n_i$, the results $h_{i-1},h_i$ of evaluating the support functions of the convex polygon $\Pi$ in the $\mathbf n_{i-1}$- and $\mathbf n_i$-directions at $\mathbf x$. The determinant of the outgoing unit normals in the numerator of \eqref{E:CRVFRMwi}, which represents the area of the shaded parallelogram in Figure~\ref{F:WxPlrCrv}, may be regarded as a discrete measure of the curvature of the boundary of $\Pi$ at the vertex $\mathbf v_i$.

\subsubsection{Wachspress coordinates on polygon boundaries}\label{SSS:WchspsBd}

Wachspress weights $w_i(\mathbf x)$ given by \eqref{E:wihihi-1} may be rescaled by $h_1\cdot\ldots\cdot h_n$ to
\begin{equation}\label{E:rescalwi}
w_i'(\mathbf x)
=h_1\cdot\ldots\cdot h_{i-2}
\cdot\det[\mathbf n_{i-1},\mathbf n_{i}]
\cdot h_{i+1}\cdot\ldots\cdot h_n
\end{equation}
as noted in Remark~\ref{R:WxprsPlr}. Suppose that an edge $\braket{\mathbf v_{j-1},\mathbf v_j}$ is not incident with $\mathbf v_i$. This means that $i\ne j-1$ and $i\ne j$, so $j-1\notin\set{i-1,i}$. If the point $\mathbf x$ lies on the edge $\braket{\mathbf v_{j-1},\mathbf v_j}$, then $h_{j-1}=0$, and the expression \eqref{E:rescalwi} for $w_i'(\mathbf x)$ vanishes. In this case, the non-zero Wachspress coordinates for $\mathbf x$ reduce to its barycentric coordinates in terms of the generators $\mathbf v_{j-1}$ and $\mathbf v_j$ of the one-dimensional simplex represented by the edge $\braket{\mathbf v_{j-1},\mathbf v_j}$. In particular, we have $w(\mathbf v_j,\mathbf v_i)=\delta_{ji}$ with the Kronecker delta function (discrete Dirac delta function). Summarizing:

\begin{proposition}
Consider a polygon $\Pi$ with vertex set $V=\set{\mathbf v_1,\dots,\mathbf v_n}$. The Wachspress weights given by \eqref{E:WaWtWaCo} may be extended, using the formula
\begin{equation}\label{E:CoWWWaCo}
w(\mathbf x,\mathbf v_i)=
\frac
{w_i'(\mathbf x)}
{\sum_{j=1}^nw_j'(\mathbf x)}\,,
\end{equation}
to the entire polygon $\Pi$.
\end{proposition}

\subsection{Wachspress coordinates for polytopes}\label{SS:WaCoFoPo}

\subsubsection{Wachspress coordinates for simple polytopes}\label{SSS:WaCoSiPo}

In \cite[\S2]{WarSchHirDes}, a method to extend Wachspress coordinates from polygons in the plane to simple polytopes in higher dimensions is presented. A $d$-dimensional polytope is said to be \emph{simple} (cf. \cite[p.8]{Ziegler}) if each vertex is incident to precisely $d$ edges. In dimension $2$, all polygons are simple.
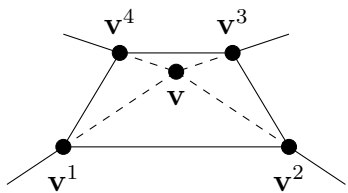
\begin{figure}[hbt]
  \centering
  \begin{tikzpicture}
    \coordinate (Origin)   at (0,0);
    \coordinate (XAxisMin) at (-.25,0);
    \coordinate (XAxisMax) at (5,0);
    \coordinate (YAxisMin) at (0,-.25);
    \coordinate (YAxisMax) at (0,2);
        \clip (-.5,-.5) rectangle (5.1cm,2.5cm);
\node[label=below:{$\mathbf v^1$}, draw, circle, inner sep=2pt, fill] at (.75,.5) {};
\node[label=below:{$\mathbf v^2$}, draw, circle, inner sep=2pt, fill] at (3.75,.5) {};
\node[label=above:{$\mathbf v^3$}, draw, circle, inner sep=2pt, fill] at (3,1.75) {};
\node[label=above:{$\mathbf v^4$}, draw, circle, inner sep=2pt, fill] at (1.5,1.75) {};
\node[label=below:{$\mathbf v$}, draw, circle, inner sep=2pt, fill] at (2.25,1.5) {};
\draw[thin] (0,0) -- (.75,.5);
\draw[thin] (4.5,0) -- (3.75,.5);
\draw[thin] (3.75,2) -- (3,1.75);
\draw[thin] (.75,2) -- (1.5,1.75);
\draw[thin] (3,1.75) -- (3.75,.5);
\draw[thin] (3,1.75) -- (1.5,1.75);
\draw[thin] (.75,.5) -- (1.5,1.75);
\draw[thin] (.75,.5) -- (3.75,.5);
\draw[dashed] (.75,.5) -- (2.25,1.5);
\draw[dashed] (3.75,.5) -- (2.25,1.5);
\draw[dashed] (3,1.75) -- (2.25,1.5);
\draw[dashed] (1.5,1.75) -- (2.25,1.5);
\end{tikzpicture}
\caption{Truncating a polytope to a simple polytope.}
\label{F:truncate}
\end{figure}
In the perspective view of part of a $3$-dimensional polytope in Figure~\ref{F:truncate}, inclusion of the vertex $\mathbf v$, incident with $4$ edges in dimension $3$, would violate the condition of simplicity. The violation could be removed by the illustrated process of \emph{truncation}, replacing the vertex $\mathbf v$ by the four new vertices $\mathbf v^1,\dots,\mathbf v^4$ as shown.

Consider a simple $d$-dimensional polytope $\Pi$ with vertices $\mathbf v_1,\dots\mathbf v_n$, and an interior point $\mathbf x$ of $\Pi$. The formula \eqref{E:CRVFRMwi} for the Wachspress weight of $\mathbf x$ with respect to the vertex $\mathbf v_i$ generalizes to
\begin{equation}\label{E:WaWtSiPo}
w_i(\mathbf x)
=
\frac
{
d!\cdot
\det[\mathbf n_{i_1},\dots,\mathbf n_{i_d}]
}
{
\braket{\mathbf v_i-\mathbf x|\mathbf n_{i_1}}
\dots
\braket{\mathbf v_i-\mathbf x|\mathbf n_{i_d}}
}
\end{equation}
(cf. \cite[(4)]{WarSchHirDes}). Here, the orientation of $\Pi$ is used to place a cyclic ordering $\mathbf n_{i_1},\dots,\mathbf n_{i_d}$ on the set of unit normals to the $d$ facets of $\Pi$ that are incident with $\mathbf v_i$.

\subsubsection{Wachspress coordinates for general polytopes}

Wachspress weights of an interior point of a polytope that is not simple may be defined in terms of the limit of the Wachspress weights, determined by \eqref{E:WaWtSiPo}, for simple truncations of $\Pi$ where the distances to eliminated non-simple points (e.g., the $\|\mathbf v-\mathbf v^{j}\|$ of Figure~\ref{F:truncate}) tend to zero. (See the discussion in \cite[\S2.2]{WarSchHirDes}.)

\subsubsection{Polytope boundaries}

Wachspress coordinates for polytopes extend to their boundaries, essentially along a generalization of the technique for polygons outlined in \S\ref{SSS:WchspsBd}.

\subsection{Continuous Wachspress coordinates}\label{SS:CoWachCo}

\subsubsection{Strictly convex sets in Euclidean space}

The Wachspress coordinates of an interior point of a polygon presented in \S\ref{SSS:WaWeSuFu}, together with their higher-dimensional version for polytopes in \S\ref{SS:WaCoFoPo}, have been developed to Wachspress coordinates of interior points $\mathbf x$ of strictly convex sets in \cite[\S3.1]{WarSchHirDes}. Recall that a convex subset of a $d$-dimensional Euclidean space is \emph{strictly convex} if it is compact (closed and bounded), and if each point $\mathbf v$ of the boundary lies in a single supporting hyperplane (cf. \cite[Ex.~3.3.24]{doCarmo}).  By analogy with \eqref{E:WaWtSiPo}, but removing the scalar factor $d!$ in the spirit of Remark~\ref{R:WxprsPlr}, the Wachspress weight of $\mathbf x$ with respect to a boundary point $\mathbf v$ is defined as
\begin{equation}\label{E:CntWxpWt}
w(\mathbf x,\mathbf v)=
\frac
{\kappa(\mathbf v)}
{\braket{\mathbf v-\mathbf x|\mathbf n_\mathbf v}^d}
\end{equation}
\cite[(6)]{WarSchHirDes}. Here, the numerator $\kappa(\mathbf v)$, the analogue of the determinant $\det[\mathbf n_{i_1},\dots,\mathbf n_{i_d}]$ in \eqref{E:WaWtSiPo}, is the Gaussian curvature of the boundary surface at the point $\mathbf v$.

Given the Wachspress weights \eqref{E:CntWxpWt}, the interior point $\mathbf x$ of the strictly convex set $\Sigma$ is then assigned
\begin{equation}\label{E:CntWxpCd}
\frac
{w(\mathbf x,\mathbf v)}
{\int_{\partial\Sigma}w(\mathbf x,\mathbf u)\mathrm d\mathbf u}
\end{equation}
as its continuous Wachspress coordinate with respect to the boundary point $\mathbf v$ \cite[(7)]{WarSchHirDes}. The expression
\begin{equation}\label{E:CntWxpEx}
\mathbf x=
\left(
\int_{\partial\Sigma}w(\mathbf x,\mathbf u)\mathrm d\mathbf u
\right)^{-1}
\int_{\partial\Sigma}
{w(\mathbf x,\mathbf v)}
\mathbf v
\mathrm d\mathbf v
\end{equation}
is obtained as the analogue of \eqref{E:WaConCom} \cite[\S3.2]{WarSchHirDes}.

\subsubsection{Curvature and determinants in dimension $2$}

\begin{figure}[hbt]
  \centering
  \begin{tikzpicture}
    \coordinate (Origin)   at (0,0);
    \coordinate (XAxisMin) at (0,0);
    \coordinate (XAxisMax) at (5,0);
    \coordinate (YAxisMin) at (0,0);
    \coordinate (YAxisMax) at (0,7);
        \clip (-1,0) rectangle (8cm,7cm);
\node[label=below:{$\mathbf v_i$}, draw, circle, inner sep=2pt, fill] at (2.5,4) {};
\node[label=above:{$\mathbf v_{i+1}$}, draw, circle, inner sep=2pt, fill] at (.25,3.25) {};
\node[label=above:{$\mathbf v_{i-1}$}, draw, circle, inner sep=2pt, fill] at (4.75,3.25) {};
\node[label=above:{$\theta_i$}] at (2.5,.75) {};
\coordinate (C) at (2.5,.25) {};
\coordinate (V+1) at (.25,3.25) {};
\coordinate (V) at (2.5,4) {};
\coordinate (M+1) at (1.375,3.625) {};
\coordinate (M-1) at (3.625,3.625) {};
\tkzMarkRightAngle[size=.3](V,M+1,C);
\tkzMarkRightAngle[size=.3](V,M-1,C);
\draw[thick,->] (2.5,4) -- node[below=2pt,right=1pt] {$\mathbf n_{i-1}$} (2.95,5.35);
\draw[thick,->] (2.5,4) -- node[below=1pt,left=.1pt] {$\mathbf n_{i}$} (2.05,5.35);
\draw[very thick] (3,5.5) arc (71.5651:108.4349:1.5811);
\node[label=above:{$\theta_i$}] at (2.5,4.5) {};
\node[label=above:
{
$
\begin{matrix}
\mbox{arc length }\lambda_i\\
\mbox{of piece }V_i\mbox{ of }\partial\Sigma
\end{matrix}
$
}
]
at (6,4) {};
\draw[->] (4.7,4.9) -- (3.4,4.1);
\node[label=above:
{
$
\begin{matrix}
\mbox{arc length }\theta_i\\
\mbox{on }S^1
\end{matrix}
$
}
]
at (0.2,5) {};
\draw[->] (1.5,6) -- (2.4,5.8);
\draw[dashed] (.25,3.25) -- (1.375,3.625);
\draw[thin] (1.375,3.625) -- (2.5,4);
\draw[dashed] (4.75,3.25) -- (3.625,3.625);
\draw[thin] (3.625,3.625) -- (2.5,4);
\draw[thin] (1.175,4.225) -- (2.5,.25);
\draw[thin] (3.75,4.015) -- (2.5,.25);
\draw (5.1517,2.9017) arc (45:135:3.75);
\draw[very thick] (3.6859,3.8076)  arc (71.5651:108.4349:3.75);
\end{tikzpicture}
\caption{Curvature and normals to inscribed polygon edges.}
\label{F:curvadet}
\end{figure}
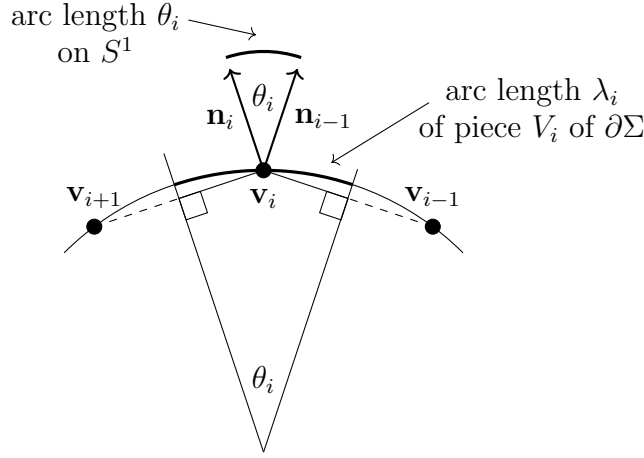
The analogy between curvature and determinants mentioned in connection with \eqref{E:CntWxpWt} may be observed most simply in dimension $2$, on the basis of Figure~\ref{F:curvadet}. On the boundary $\partial\Sigma$ of a strictly convex set $\Sigma$, consider a point $\mathbf v_i$ that is also a member of the large set $\set{\mathbf v_1,\dots,\mathbf v_n}$ of vertices of a polytope $\Pi$ that is inscribed in, and closely approximates, the convex set $\Sigma$. Suppose that the successive normals $\mathbf n_{i-1},\mathbf n_i$ to the edges of $\Pi$ that contain $\mathbf v_i$ subtend an angle $\theta_i$, as shown in Figure~\ref{F:curvadet}. Let $V_i$ denote the corresponding portion of the boundary $\partial\Sigma$, with arc length $\lambda_i$. In Figure~\ref{F:curvadet}, $S^1$ denotes the unit circle centered at $\mathbf v_i$.

In the discrete analysis of the Wachspress coordinates with respect to $\Pi$, involving functions $f$ defined on the vertex set $\set{\mathbf v_1,\dots,\mathbf v_n}$, functionals of the form
\begin{equation}\label{E:disumfnc}
\sum_{i=1}^nf(\mathbf v_i)\det[\mathbf n_{i-1},\mathbf n_{i}]
\end{equation}
appear. For example, we have $f(\mathbf v_i)$ as
\begin{equation}\label{E:discrfns}
\frac
{2!}
{\braket{\mathbf v_i-\mathbf x|\mathbf n_{i-1}}
\braket{\mathbf v_i-\mathbf x|\mathbf n_i}}
\quad
\mbox{or}
\quad
\frac
{2!\mathbf v_i}
{\braket{\mathbf v_i-\mathbf x|\mathbf n_{i-1}}
\braket{\mathbf v_i-\mathbf x|\mathbf n_i}}
\end{equation}
appearing in the version of \eqref{E:WaConCom} that uses Wachspress weights given by specializing the general formula \eqref{E:WaWtSiPo} to the case $d=2$.

Now consider an integrable function $f$ defined on $\partial\Sigma$. For the analogues of \eqref{E:discrfns}, we might have $f(\mathbf v)$ as $w(\mathbf x,\mathbf v)$ or $w(\mathbf x,\mathbf v)\mathbf v$ with the notation of \eqref{E:CntWxpWt}. We may then regard the functional \eqref{E:disumfnc} as a Riemann sum, determined by the partition $V_1,\dots,V_n$ of $\partial\Sigma$, to approximate the integral functional
\begin{equation}\label{E:curvfctl}
\int_{\partial\Sigma}f(\mathbf v)\kappa(\mathbf v)ds
\end{equation}
involving the arc length parameter $s$ for $\partial\Sigma$. Indeed, \eqref{E:curvfctl} is
\begin{align*}
\sum_{i=1}^n&\int_{V_i}f(\mathbf v)\kappa(\mathbf v)ds
\simeq\sum_{i=1}^n
\int_{V_i}f(\mathbf v_i)\frac{\theta_i}{\lambda_i}ds
\simeq\sum_{i=1}^n
\int_{V_i}f(\mathbf v_i)\frac{\sin\theta_i}{\lambda_i}ds
\\
&
=\sum_{i=1}^nf(\mathbf v_i)\sin\theta_i
\cdot\frac{1}{\lambda_i}\int_{V_i}1ds
=\sum_{i=1}^nf(\mathbf v_i)
\det[\mathbf n_{i-1},\mathbf n_{i}] \,.
\end{align*}
Here, the first approximation follows by \cite[p.167, Remark]{doCarmo}, for example. Intuitively, $V_i$ is almost an arc of angle $\theta_i$ and radius $r_i=\kappa(\mathbf v_i)^{-1}$, so $\lambda_i\simeq r_i\theta_i$.

\section{Comparing Wachspress and Gibbs coordinates}\label{S:WchsGibs}

\subsection{When Wachspress and Gibbs coordinates coincide}\label{SS:WaGiCncd}

\begin{proposition}\label{P:W=Gsmplx}
Consider a simplex as a free barycentric algebra $XB$ on its set $X$ of extreme points. Consider the insertion
\begin{equation}\label{E:W=Gsmplx}
j\colon X\hookrightarrow XB
\end{equation}
as a valuation function. Then on $XB$, Wachspress coordinates agree with the Gibbs coordinates for the valuation function \eqref{E:W=Gsmplx}.
\end{proposition}

\begin{definition}\label{D:semsmplx}
A polytope is said to be a \emph{semisimplex} if, as a barycentric algebra, it is a direct product of simplices (free barycentric algebras, as in \S\ref{SSS:fralasmp}).
\end{definition}

\begin{remark}
The only semisimplices in the plane are parallelograms and triangles. Recall from \S\ref{SSS:WaCoSiPo} that all polygons are simple, so that notion of simplicity is not related to the semisimplicity of Definition~\ref{D:semsmplx}.
\end{remark}

\begin{theorem}\label{T:aditivty}\cite[Th.~15.1]{Sm151}
For $i=1,2$, consider constituent systems represented by valuation functions $f_i:X_i\to A_i$, with respective entropy functions $H_i$. Let $f:X\to A$ be the valuation function
$$
f_1\times f_2:X_1\times X_2\to A_1\times A_2
$$
of the independent combination of the constituent systems, with entropy function $H$. Then
$$
H\big((\alpha_1,\alpha_2)\big)
= H_1(\alpha_1)+H_2(\alpha_2)
$$
for elements $\alpha_i$ of $A_i$.
\end{theorem}

\begin{theorem}\label{T:WaGiSmSi}
Wachspress coordinates on semisimplices coincide with the Gibbs coordinates.
\end{theorem}

\begin{proof}
The Wachspress coordinates on a semisimplex are given as products of the Wachspress coordinates on its simplicial direct factors \cite{WarrenUBC}. (Note that Warren actually speaks of ``tensor products'' of simplices --- compare \S\ref{SSS:dipotepo}.) Now by Proposition~\ref{P:W=Gsmplx}, Wachspress coordinates maximize the entropy on simplices, since they coincide on a simplex with the entropy-maximizing Gibbs coordinates. By Theorem~\ref{T:aditivty}, Wachspress coordinates also maximize entropy on semisimplices. It follows that they coincide there with the entropy-maximizing Gibbs coordinates.
\end{proof}

\subsection{When Wachspress and Gibbs coordinates differ}\label{SS:WxNeqGbs}

In general, the Wachspress coordinates of a point in a convex polygon may differ from its Gibbs coordinates.

\begin{example}\label{X:WxNeqGbs}
Let $A$ be the convex hull (in $\mathbb{R}^2$) of the ordered sequence
$$
\mathbf v_1=(0,0)
=\tfrac23\mathbf v_2-\tfrac13\mathbf v_3+\tfrac23\mathbf v_4,
\mathbf v_2=(1,0),
\mathbf v_3=(0,1),
\mathbf v_4=(-1,\tfrac{1}{2})
$$
of extreme points. (Compare Figure~\ref{F:QrCrtGbs}.) Note that $A$ consists of an isosceles triangle based on the left of the $y$-axis, abutted to a right triangle based on the right of the $y$-axis. Thus
\begin{align*}
\mathbf b
&
=\tfrac16\left(\mathbf v_1+\mathbf v_3+\mathbf v_4\right)
+\tfrac16\left(\mathbf v_1+\mathbf v_2+\mathbf v_3\right)
\\
&
=\tfrac12\left(-\tfrac13,\tfrac12\right)
+\tfrac12\left(\tfrac14,\tfrac14\right)
=\left(0,\tfrac5{12}\right)
\simeq(0,0.417)
\end{align*}
is the barycenter of $A$.

Take $V=\set{\mathbf v_1,\mathbf v_2,\mathbf v_3,\mathbf v_4}$ with the insertion $j\colon V\hookrightarrow A$ as a valuation function. Corresponding to the barycentric homomorphism specified (to three decimal places) by
$$
\beta\colon A\to\mathbb R^\infty;
\mathbf v_2\,,
\mathbf v_4\mapsto\log 4+0.041\,,
\mathbf v_3\mapsto\log 4-0.172\,,
$$
the vector of Gibbs coordinates (to $3$ decimal places) for the barycenter $\mathbf b$ is obtained as $(0.223,0.240,0.297,0.240)$, while
$$
(0.186,0.265,0.284,0.265)=\tfrac1{264}(49,70,75,70)
$$
presents its Wachspress coordinates.

Again, we have a point
\begin{equation}\label{E:pointamx}
\mathbf a=\sum_{i=1}^4\tfrac14\mathbf v_i
=\left(0,\tfrac38\right)
=(0,0.375)
\end{equation}
realized from the Gibbs distribution for the potential $\beta_0\colon A\to\mathbb R^\infty$ with constant value $0$ (as in Proposition~\ref{P:GaGrOrCs}). However,
\begin{align*}
\tfrac1{4\times28}({25},{30},{27},{30})
&
=
(\tfrac14,\tfrac14,\tfrac14,\tfrac14)
+\tfrac1{112}(-3,2,-1,2)
\\
&
\simeq
(0.223,0.268,0.241,0.268)
\end{align*}
gives the Wachspress coordinate vector for $\mathbf a$. Here, we define
$$
\tfrac1{112}(3,-2,1)\simeq(0.027,-0.018,0.009)
$$
to be the (\emph{G$-$W}) \emph{discrepancy vector} of $\mathbf a$. [Read as ``G minus W''.] For the barycenter $\mathbf b$, the discrepancy vector is approximately $(0.037,-0.025,0.013)$.
\end{example}

\begin{figure}[thb]
\centering
\includegraphics[width=0.95\textwidth]{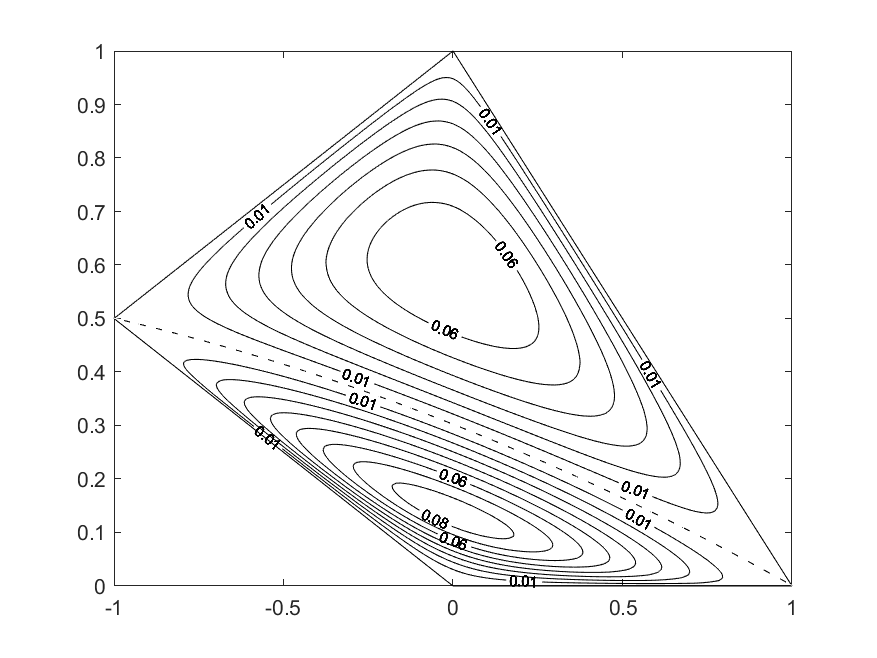}
\caption{The norm of the G$-$W discrepancy vector for the quadrilateral $A$ of \protect{Example~\ref{X:WxNeqGbs}}, displayed as a contour plot. The dotted line shows the \emph{equator} (compare \S\ref{SS:EquaEqua}) where the Gibbs and Wachspress coordinates agree; i.e., where the norm is zero.}
\label{F:QrCrtGbs}
\end{figure}

\subsection{The G$-$W discrepancy field}\label{SS:GWdiscrp}

Consider a polygon in the plane, with a vertex set of cardinality $n$. We introduce an $(n-1)$-dimensional vector field over the polygon to track the discrepancies between Gibbs coordinates and Wachspress coordinates.

\begin{definition}\label{D:GWdscrpFld}
Consider a polygon $\Pi$ presented as the convex hull of the counter-clockwise ordered sequence $\mathbf v_1,\dots,\mathbf v_n$ of extreme points around its boundary. For Gibbs coordinates, consider the insertion
$$
j\colon\set{\mathbf v_1,\dots,\mathbf v_n}\hookrightarrow\Pi;
\mathbf v_i\mapsto\mathbf v_i
$$
of the extreme points into the polygon as the valuation function. Then the (\emph{G$-$W}) \emph{discrepancy field} of $\Pi$ is the $(n-1)$-dimensional vector field over $\Pi$ whose $i$-th coordinate
\begin{equation}\label{E:GWdscrpFld}
q^\beta\left(\mathbf x,\mathbf v_i\right)
-
w(\mathbf x,\mathbf v_i)
\end{equation}
for $1\le i<n$, at a point $\mathbf x$ of $\Pi$, expresses the difference between the $i$-th Gibbs coordinate given by Definition~\ref{D:Gibalpha} and the corresponding Wachspress coordinate given by Definition~\ref{D:WxprsPlr}.
\end{definition}

\begin{proposition}\label{P:GWdscrpFld}
In the context of Definition~\ref{D:GWdscrpFld}, the G--W discrepancy vectors of $\Pi$ lie in a vector space of dimension $n-3$.
\end{proposition}

\begin{proof}
In a Cartesian coordinatization of the plane, suppose that each vertex $\mathbf v_i$ has coordinates $(v_i^1,v_i^2)$, for $1\le i\le n$. Then the relations
$$
\sum_{i=1}^nq^\beta\left(\mathbf x,\mathbf v_i\right)\mathbf v_i
=
\mathbf x
=
\sum_{i=1}^nw(\mathbf x,\mathbf v_i)\mathbf v_i
$$
imply the two independent linear conditions
$$
\sum_{i=1}^n
\left(
q^\beta\left(\mathbf x,\mathbf v_i\right)
-w(\mathbf x,\mathbf v_i)
\right)v_i^j=0
$$
for $j=1,2$. These conditions determine a subspace of codimension $2$ in the $(n-1)$-dimensional space provided by Definition~\ref{D:GWdscrpFld}.
\end{proof}

\begin{corollary}\label{C:GWdscrpFld}
The (non-vanishing) $G-W$ discrepancy vectors of various points of a quadrilateral are parallel.
\end{corollary}

\subsection{The equator equations}\label{SS:EquaEqua}

Within the context of the quadrilateral $A$ of Example~\ref{X:WxNeqGbs} and Figure~\ref{F:QrCrtGbs}, we consider the \emph{equator}: the curve of points connecting the vertices $(-1,0.5)$ and $(0,1)$ through the interior of $A$ where the Gibbs and Wachspress coordinates coincide.

\subsubsection{Gibbs coordinates}\label{E:equaGibs}

The Gibbs coordinates of an interior point of the quadrilateral are given by a barycentric homomorphism
\begin{align}\notag
&\beta\colon A\to\mathbb R^\infty;
\\ \label{E:beta4equ}
&
\begin{cases}
\mathbf v_1\mapsto\log\sqrt[3]{x^2y^{-1}z^2}\,,
\\
\mathbf v_2\mapsto\log x\,,
\\
\mathbf v_3\mapsto\log y,
\\
\mathbf v_4\mapsto\log z
\end{cases}
\end{align}
for positive real numbers $x,y,z$ that are to be regarded as unknowns. Note
$$
\mathbf v_1\beta
=\tfrac23\mathbf v_2\beta-\tfrac13\mathbf v_3\beta+\tfrac23\mathbf v_4\beta,
$$
so that \eqref{E:beta4equ} indeed specifies a barycentric algebra homomorphism. Write the partition function $Z(\beta)$ as the algebraic function
$$
Z=Z(x,y,z)=x+y+z+\sqrt[3]{x^2y^{-1}z^2}
$$
of the positive real variables $x,y,z$. Then with $w=\sqrt[3]{x^2y^{-1}z^2}$, we have
\begin{equation}\label{E:GibsEqEq}
\left(\frac{\sqrt[3]{x^2y^{-1}z^2}}Z,\frac xZ,\frac yZ,\frac zZ\right)
=
\left(\frac wZ,\frac xZ,\frac yZ,\frac zZ\right)
\end{equation}
as the vector of Gibbs coordinates of the point determined by \eqref{E:beta4equ}, while
\begin{align}\notag
&\frac{\sqrt[3]{x^2y^{-1}z^2}}Z(0,0)
+\frac xZ(1,0)
+\frac yZ(0,1)
+\frac zZ(-1,\tfrac12)
=
\\ \label{E:CartEqEq}
&
\rule{10mm}{0mm}
\mathbf x
:=\left(\frac{2x-2z}{2Z},\frac{2y+z}{2Z}\right)
\end{align}
gives the Cartesian coordinates of that point.

\subsubsection{Wachspress coordinates}

We determine the Wachspress coordinates of the Cartesian point \eqref{E:CartEqEq} using Proposition~\ref{P:WxprsPlr}. Recalling \eqref{E:Aov0v1v2}, and with $\mathbf v_i=(v_{i,1},v_{i,2})$, then after a little manipulation for the penultimate equation we obtain
\renewcommand{\arraystretch}{1.5}
\begin{align*}
A(\mathbf x,&\mathbf v_i,\mathbf v_{i+1})
=\tfrac12\det[\mathbf v_i-\mathbf  x,\mathbf v_{i+1}-\mathbf  x]
\\
&
=
\frac12
\begin{vmatrix}
v_{i,1}-\frac{2x-2z}{2Z} &v_{i,2}-\frac{2y+z}{2Z}\\
v_{{i+1},1}-\frac{2x-2z}{2Z} &v_{{i+1},2}-\frac{2y+z}{2Z}
\end{vmatrix}
\\
&
=\frac1{4Z^2}
\big\{
(Zv_{i,1}-x+z)(2Zv_{{i+1},2}-2y-x)
\\
&
\rule{15mm}{0mm}
-(Zv_{{i+1},1}-x+z)(2Zv_{i,2}-2y-z)
\big\}
\end{align*}
for $1\le i\le 4$, with $\mathbf v_{i+1}$ interpreted as $\mathbf v_1$ for $i=4$. We then have the following expressions:
\begin{align}\label{E:s1Whsprs}
s_1:=4ZA(\mathbf x,\mathbf v_1,\mathbf v_2)&=(x+2y)\,;\\ \label{E:s2Whsprs}
s_2:=4ZA(\mathbf x,\mathbf v_2,\mathbf v_3)&=(2Z-3x-2y+2z)\,;\\ \label{E:s3Whsprs}
s_3:=4ZA(\mathbf x,\mathbf v_3,\mathbf v_4)&=(2Z-2y-z)\,;\mbox{ and}\\ \label{E:s4Whsprs}
s_4:=4ZA(\mathbf x,\mathbf v_4,\mathbf v_1)&=(2x+2y-z)\,.
\end{align}
By \eqref{E:WxprsPlr}, the Wachspress weights become:
\begin{align*}
w_1(\mathbf x)
&
=
\frac
{
A\left(
\mathbf v_{4},\mathbf v_{1},\mathbf v_{2}
\right)
}
{
A\left(
\mathbf x,\mathbf v_{4},\mathbf v_{1}
\right)
\cdot
A\left(
\mathbf x,\mathbf v_{1},\mathbf v_{2}
\right)
}
=
\frac
{
4Z^2
}
{
(2x+2y-z)(x+2y)
}\,;
\\
w_2(\mathbf x)
&
=
\frac
{
A\left(
\mathbf v_{1},\mathbf v_{2},\mathbf v_{3}
\right)
}
{
A\left(
\mathbf x,\mathbf v_{1},\mathbf v_{2}
\right)
\cdot
A\left(
\mathbf x,\mathbf v_{2},\mathbf v_{3}
\right)
}
=
\frac
{
8Z^2
}
{
(x+2y)(2Z-3x-2y+2z)
}\,;
\\
w_3(\mathbf x)
&
=
\frac
{
A\left(
\mathbf v_{2},\mathbf v_{3},\mathbf v_{4}
\right)
}
{
A\left(
\mathbf x,\mathbf v_{2},\mathbf v_{3}
\right)
\cdot
A\left(
\mathbf x,\mathbf v_{3},\mathbf v_{4}
\right)
}
=
\frac
{
12Z^2
}
{
(2Z-3x-2y+2z)(2Z-2y-z)
}\,;
\\
w_4(\mathbf x)
&
=
\frac
{
A\left(
\mathbf v_{3},\mathbf v_{4},\mathbf v_{1}
\right)
}
{
A\left(
\mathbf x,\mathbf v_{3},\mathbf v_{4}
\right)
\cdot
A\left(
\mathbf x,\mathbf v_{4},\mathbf v_{1}
\right)
}
=
\frac
{
8Z^2
}
{
(2Z-2y-z)(2x+2y-z)
}\,.
\end{align*}
The inverted term in \eqref{E:WaConCom} is
\begin{align*}
w_1(\mathbf x)&+w_2(\mathbf x)+w_3(\mathbf x)+w_4(\mathbf x)
\\
&
=4Z^2
\left(
\frac1{s_4s_1}
+
\frac2{s_1s_2}
+
\frac3{s_2s_3}
+
\frac2{s_3s_4}
\right)\\
&
=
4Z^2
\left(
\frac
{s_2s_3+2s_3s_4+3s_4s_1+2s_1s_2}
{s_1s_2s_3s_4}
\right)\,.
\end{align*}
Set
$$
S=s_2s_3+2s_3s_4+3s_4s_1+2s_1s_2\,.
$$
Then
\begin{equation}\label{E:WhspEqEq}
\left(
\frac{s_2s_3}S,\frac{2s_3s_4}S,\frac{3s_4s_1}S,\frac{2s_1s_2}S
\right)
\end{equation}
describes the vector of Wachspress coordinates of the interior point $\mathbf x$ in terms of the expressions \eqref{E:s1Whsprs}--\eqref{E:s4Whsprs}.
\renewcommand{\arraystretch}{1}

\subsubsection{Equator equations}\label{SSS:EquaEqua}

Comparing the last three coordinates of \eqref{E:GibsEqEq} against the last three coordinates of \eqref{E:WhspEqEq}, we obtain the sets
\begin{equation*}
\frac{2s_3s_4}S=\frac xZ\,,
\quad
\frac{3s_4s_1}S=\frac yZ\,,
\quad
\frac{2s_1s_2}S=\frac zZ
\end{equation*}
or
\begin{align}\label{E:EquaEqua}
2s_3s_4Z=xS\,,
\quad
3s_4s_1Z=yS\,,
\quad
2s_1s_2Z=zS
\end{align}
of three equations which hold on the equator.

In general, the set of points where Gibbs and Wachspress coordinates agree is given as the solution set of a collection of equations. This solution set may comprise the entire polygon, as treated in \S\ref{SS:WaGiCncd}. However, in the  case discussed here, the solution set in the interior of the quadrilateral forms an equator line whose equation is determined explicitly in Theorem~\ref{T:EqtrEqun} below.

\subsubsection{The equator}

The Cartesian coordinates of a point of $A$ were given as
\begin{equation}\label{E:abfromG}
(a,b)=\left(\frac{x-z}{Z},\frac{2y+z}{2Z}\right)
\end{equation}
by \eqref{E:CartEqEq} in terms of the \emph{Gibbs variables}
$$
x,y,z,Z=x+y+z+\sqrt[3]{x^2y^{-1}z^2}\,.
$$
Similarly, we may compute
\begin{equation}\label{E:abfromW}
(a,b)=\left(\frac{2(uv-st)}{S},\frac{s(t+3v)}{S}\right)
\end{equation}
in terms of the \emph{Wachspress variables}
$$
s=s_1,t=s_2,u=s_3,v=s_4,S=s_2s_3+2s_3s_4+3s_4s_1+2s_1s_2\,.
$$

Consider points $(a,b)$ of the quadrilateral $A$, as in \eqref{E:abfromG} and\eqref{E:abfromW}. We know that the Gibbs and Wachspress coordinates of $(a,b)$ agree at $\mathbf v_1=(0,0)$ and $\mathbf v_3=(0,1)$. Otherwise, we have the following parametric equation of the equator from $\mathbf v_4=(-1,1/2)$ to $\mathbf v_2=(1,0)$.

\begin{theorem}\label{T:EqtrEqun}
The point $(a,b)$ lies on the equator if
\begin{equation}\label{E:EqtrEqun}
b=
\frac
{
-12-13a+
\sqrt{25^2+11\cdot13(1-a^2)}
}
{4\cdot13}
\end{equation}
with $-1\le a\le 1$.
\end{theorem}

\begin{proof}
Setting $a=\pm1$ in \eqref{E:EqtrEqun} returns the vertices $\mathbf v_4=(-1,1/2)$ and $\mathbf v_2=(1,0)$ as equator points $(a,b)$. For the remainder of the proof, we restrict to the case where $-1<a<1$, considering interior points of the quadrilateral. This implies that $x,y,z$ and $w=\sqrt[3]{x^2y^{-1}z^2}$ are positive real numbers.

Comparing \eqref{E:GibsEqEq} with \eqref{E:WhspEqEq} yields
\begin{align*}
\frac{3xz}{Z^2}
=\frac{12s_1s_2s_3s_4}{S^2}
=\frac{4wy}{Z^2}
\end{align*}
or
\begin{equation}\label{E:Equa_imp}
3xz=4wy\,,
\end{equation}
whence $x^2z^2/y=16w^2y/9$ and
\begin{align*}
\sqrt[3]{w^2\cdot w}=w=\sqrt[3]{x^2y^{-1}z^2}=\sqrt[3]{\frac{16}{9}w^2y}\,,
\end{align*}
implying that $w=16y/9$. Taking $x=aZ+z$ from the first component of \eqref{E:abfromG}, and
$$
Z=x+z+y+w=aZ+2z+25y/9\,,
$$
we have
\begin{align}\label{E:ZandxzaZ}
Z&=\frac{2z+25y/9}{1-a}
\quad
\mbox{and}
\quad
x=z+aZ
=\frac{z(1+a)+25ay/9}{1-a}\,.
\end{align}
The equation
\begin{equation*}
\frac{3(z^2(1+a)+25azy/9)}{1-a}=\frac{64}{9}y^2
\end{equation*}
is then obtained from \eqref{E:Equa_imp}. Dividing throughout by $y^2$, and setting $\xi=\dfrac{z}{y}$, we have
\begin{equation*}
3\cdot\frac{1+a}{1-a}\cdot\xi^2+\frac{25a}{3\cdot(1-a)}\cdot\xi-\frac{64}{9}=0.
\end{equation*}
The unknown $\xi=z/y$ satisfies
$$
\xi=\frac{-25a+\sqrt{768-143a^2}}{18(1+a)}\,,
$$
as the positive solution of the quadratic equation.
The second component of \eqref{E:abfromG} now yields
\begin{align*}
b&=\frac{y+z/2}{Z}
\overset{\eqref{E:ZandxzaZ}}{=}
\frac{y+\xi y/2}{y(2\xi+25/9)/(1-a)}=\frac{9(1-a)(1+\xi/2)}{18\xi+25}\\
&=\frac14(1-a)
\frac
{36+11a+\sqrt{768-143a^2}}
{25+\sqrt{768-143a^2}}
\\
&=
\frac14(1-a)
\frac
{36+11a+\sqrt{768-143a^2}}
{25+\sqrt{768-143a^2}}
\cdot
\frac
{25-\sqrt{768-143a^2}}
{25-\sqrt{768-143a^2}}
\\
&=
\frac
{22(1+a)\sqrt{768-143a^2}-264-550a-286a^2}
{1144(1+a)}
\\
&=
\frac
{\sqrt{768-143a^2}}
{52}
-
\frac
{12+25a+13a^2}
{52(1+a)}
=
\frac
{\sqrt{768-143a^2}}
{52}
-
\frac
{12+13a}
{52}
\\
&=
\frac
{
-12-13a+
\sqrt{25^2+11\cdot13(1-a^2)}
}
{4\cdot13}
\end{align*}
as required.
\end{proof}

\section{Conclusion and future work}\label{S:ConcFuWk}

\subsection{Concluding remarks}

We have applied the techniques of barycentric algebra to the problem of determining barycentric coordinates for elements of a polytope. Theorem~\ref{T:WaGiSmSi} identifies a class of polytopes on which the Gibbs and Wachspress coordinates agree. The G--W discrepancy vector is introduced as a tool for comparing Gibbs and Wachspress coordinates on a given polygon.

\subsubsection{Continuous Wachspress coordinates}

The standard approach to these coordinates was surveyed briefly in \S\ref{SS:CoWachCo}. It may prove fruitful to adopt an alternative approach based on barycentric algebras, making use of the modal theory in \cite[Ch.~3]{RS85}. This theory formalizes the approximation of a convex set by the polytopes it contains, or more generally realizing an arbitrary barycentric algebra as a colimit of its finitely generated subalgebras.

\subsubsection{Threshold convexity}

We have not addressed numerical issues in this paper. One such issue is the important distinction between zero and non-zero coordinates. In dealing with this issue, it may be helpful to consider \emph{threshold convexity}, a version of barycentric algebra effectively defaulting to zero or one for coordinates in weighted means whose closeness to one of the endpoints of the unit interval lies within the level of numerical noise \cite{Sm184}.

\end{document}